\documentclass[notitlepage,11pt,reqno]{amsart}
\usepackage[foot]{amsaddr}
\usepackage[numbers,sort]{natbib}

\usepackage{amssymb,nicefrac,bm,upgreek,mathtools,verbatim}
\usepackage[final]{hyperref}
\usepackage[mathscr]{eucal}
\usepackage{dsfont}

\usepackage{color}
\definecolor{dmagenta}{rgb}{.4,.1,.5}
\definecolor{dblue}{rgb}{.0,.0,.5}
\definecolor{mblue}{rgb}{.0,.0,.8}
\definecolor{ddblue}{rgb}{.0,.0,.4}
\definecolor{dred}{rgb}{.6,.0,.0}
\definecolor{dgreen}{rgb}{.0,.5,.0}
\definecolor{Eeom}{rgb}{.0,.0,.5}

\usepackage[normalem]{ulem}
\usepackage[margin=1in]{geometry}
\allowdisplaybreaks
\newtheorem{lemma}{Lemma}[section]
\newtheorem{theorem}{Theorem}[section]

\newtheorem{corollary}{Corollary}[section]

\theoremstyle{definition}
\newtheorem{definition}{Definition}[section]

\theoremstyle{remark}

\newtheorem{remark}{Remark}[section]

\numberwithin{equation}{section}

\hypersetup{
  colorlinks=true,
  citecolor=dblue,
  linkcolor=dblue,
  frenchlinks=false,
  pdfborder={0 0 0},
  naturalnames=false,
  hypertexnames=false,
  breaklinks}
\usepackage[capitalize,nameinlink]{cleveref}

\crefname{section}{Section}{Sections}
\crefname{subsection}{Subsection}{Subsections}
\crefname{condition}{Condition}{Conditions}
\crefname{hypothesis}{Hypothesis}{Conditions}
\crefname{assumption}{Assumption}{Assumptions}
\crefname{lemma}{Lemma}{Lemmas}
\crefname{claim}{Claim}{Claims}

\Crefname{figure}{Figure}{Figures}

\crefformat{equation}{\textup{#2(#1)#3}}
\crefrangeformat{equation}{\textup{#3(#1)#4--#5(#2)#6}}
\crefmultiformat{equation}{\textup{#2(#1)#3}}{ and \textup{#2(#1)#3}}
{, \textup{#2(#1)#3}}{, and \textup{#2(#1)#3}}
\crefrangemultiformat{equation}{\textup{#3(#1)#4--#5(#2)#6}}%
{ and \textup{#3(#1)#4--#5(#2)#6}}{, \textup{#3(#1)#4--#5(#2)#6}}%
{, and \textup{#3(#1)#4--#5(#2)#6}}

\Crefformat{equation}{#2Equation~\textup{(#1)}#3}
\Crefrangeformat{equation}{Equations~\textup{#3(#1)#4--#5(#2)#6}}
\Crefmultiformat{equation}{Equations~\textup{#2(#1)#3}}{ and \textup{#2(#1)#3}}
{, \textup{#2(#1)#3}}{, and \textup{#2(#1)#3}}
\Crefrangemultiformat{equation}{Equations~\textup{#3(#1)#4--#5(#2)#6}}%
{ and \textup{#3(#1)#4--#5(#2)#6}}{, \textup{#3(#1)#4--#5(#2)#6}}%
{, and \textup{#3(#1)#4--#5(#2)#6}}

\crefdefaultlabelformat{#2\textup{#1}#3}

\makeatletter
\DeclareRobustCommand\widecheck[1]{{\mathpalette\@widecheck{#1}}}
\def\@widecheck#1#2{%
    \setbox\z@\hbox{\m@th$#1#2$}%
    \setbox\tw@\hbox{\m@th$#1%
       \widehat{%
          \vrule\@width\z@\@height\ht\z@
          \vrule\@height\z@\@width\wd\z@}$}%
    \dp\tw@-\ht\z@
    \@tempdima\ht\z@ \advance\@tempdima2\ht\tw@ \divide\@tempdima\thr@@
    \setbox\tw@\hbox{%
       \raise\@tempdima\hbox{\scalebox{1}[-1]{\lower\@tempdima\box
\tw@}}}%
    {\ooalign{\box\tw@ \cr \box\z@}}}
\makeatother

\newcommand{\df}{\coloneqq}
\newcommand{\D}{\mathrm{d}}          %differential
\newcommand{\RR}{\mathbb{R}}         %Real numbers
\newcommand{\Rd}{{\mathbb{R}^d}}       %R^d
\newcommand{\NN}{\mathbb{N}}         %Natural numbers
\newcommand{\Ind}{\mathds{1}}            %indicator function
\newcommand{\sL}{\mathscr{L}}        % Extended generator
\newcommand{\sB}{\mathscr{B}}    % Ball used often
\newcommand{\cC}{\mathcal{C}}     % Classes of continuous functions
\newcommand{\cM}{\mathcal{M}}
\newcommand{\cQ}{\mathcal{Q}}   
\newcommand{\abs}[1]{\lvert#1\rvert}
\newcommand{\norm}[1]{\lVert#1\rVert}

\begin{document}

\title[Regularity results of nonlinear nonlocal operators]%
{Regularity results of nonlinear perturbed stable-like operators}

\author[Anup Biswas]{Anup Biswas}
\address{$^\dag$ Department of Mathematics,
Indian Institute of Science Education and Research,
Dr. Homi Bhabha Road, Pune 411008, India}
\email{anup@iiserpune.ac.in, mit.modasiya@gmail.com}

\author[Mitesh Modasiya]{Mitesh Modasiya}

\date{}

%%%%%%%%%%%%%%%%%%%%%%%%%%%%%%%%%%%%%%%%%%%%%%%%%%%%%%%%%%%%%%%%%%%%%%%%%%%%%%%%
\begin{abstract}
We consider a class of fully nonlinear integro-differential operators where the nonlocal integral has two
components: the non-degenerate one corresponds to the $\alpha$-stable operator and the second one
(possibly degenerate) corresponds to a class of \textit{lower order} L\'evy measures. Such operators do not
have a global scaling property. We establish H\"{o}lder regularity, Harnack inequality and
boundary Harnack property of solutions of these operators. 
\end{abstract}



\maketitle

\section{Introduction}
In this article we are concerned with the regularity property of nonlinear integro-differential 
elliptic operators of the form
\begin{equation*}
Iu=\inf_\beta\sup_\alpha L_{\alpha\beta} u(x) =\,\int_{\Rd}\bigl(u(x+y)+u(x-y)-2u(x)\bigr) \frac{k_{\alpha\beta}(y)}{\abs{y}^{d}}\, \D{y}\,,
\end{equation*}
where $k_{\alpha\beta}$ is symmetric and satisfies
$$(2- \alpha) \lambda \frac{1}{{\abs y}^{\alpha}} 
\;\leq\; k_{\alpha\beta}(y) \leq  \Lambda \left( \frac{2- \alpha}{{\abs y}^{\alpha}} + 
\varphi({1}/{\abs y}) \right)\,,\quad 0<\lambda\leq\Lambda\,,$$
for some function $\varphi:(0, \infty)\to (0, \infty)$ satisfying a weak upper scaling property with exponent
$\beta<\alpha$ (see \cref{S-model} for a precise definition). Such operators are of great importance in control
theory and were first considered by Pucci \cite{P66} to study the principal eigenvalue problem for local
nonlinear elliptic operators. In a series of influential works \cite{CS09,CS11,CS11a} Caffarelli and Silvestre
develop a regularity theory of nonlinear stable-like integro-differential operator with symmetric
kernels. Optimal boundary regularity for such operators are established by Serra and Ros-Oton in
\cite{RS16,RS17} whereas the boundary Harnack property is considered in \cite{RS19}.
Kriventsov in \cite{K13} studies interior $\cC^{1, \gamma}$ regularity for rough symmetric kernel and
his result is further improved by Serra in \cite{S15} who establishes interior $\cC^{\alpha+\gamma}$ estimate
with rough symmetric kernels.

There is also an extensive amount of work extending the results of Caffarelli and Silvestre \cite{CS09,CS11,CS11a}. In \cite{KKL16} the authors generalized these results to fully nonlinear 
integro-differential operators with regularly varying kernels.
Regularity results for nonsymmetric stable-like kernels are studied in \cite{CLD,KL12}.
Recently, \cite{KL20} generalize these results for kernels with variable orders. These kernels
are closely related to an important family of L\'evy processes known as subordinate Brownian motions.
Subordinate Brownian motions(sBM) are obtained by time-changing the Brownian motion by an independent 
subordinator (i.e.,\ nondecreasing, non-negative L\'evy process). In particular, when
the subordinator is $\alpha$-stable we obtain a $\alpha$-stable process as sBM whose generator
is given by the fractional Laplacian. 
Some further insightful discussion and comparison of \cite{KL20} with the current work are left for 
\cref{S-model}. Let us also mention \cite{JB15,KM17} which also study regularity results for similar model.
Our current work is closely related to a recent work of Mou \cite{M19} where the author considers a 
nonlinear second order uniformly elliptic integro-differential PDE and establishes H\"{o}lder regularity of the solution. Our work should be seen as a nonlocal counterpart of \cite{M19} where the H\"{o}lder regularity
follows due to the non-degeneracy of $\alpha$-stable kernel. We also consider Harnack inequality and boundary
Harnack property for operators that are elliptic with respect to the family of linear operators having kernel
function $k$ satisfying
$$ \lambda (\frac{(2- \alpha)}{{\abs y}^{\alpha}}+  \varphi({1}/{\abs y}))
\;\leq\; k(y) \leq  \Lambda \left( \frac{2- \alpha}{{\abs y}^{\alpha}} + 
\varphi({1}/{\abs y}) \right)\,.$$

The rest of the article is organized as follows: In the next section we introduce the model 
and provide a motivation for considering such model.
\cref{S-ABP} contains the proofs of ABP estimate(\cref{T3.1})
and weak-Harnack inequality(\cref{T3.4}). In \cref{S-Holder} we study  the interior H\"{o}lder regularity(\cref{T4.1}) 
and the Harnack inequality(\cref{T5.1}) is established in \cref{S-Har}. Finally, in \cref{S-bHar} we prove a boundary Harnack estimate(\cref{T6.1}).
%%%%%%%%%%%%%%%%%%%%%%%%%%%%%%%%%%%%%%%%%%%%%%%%%%%%%%%%%%%%%%%%%%%%%%%%%%%%%%%%%%%%%%%%%%%%%%%%%%%%%%%%
\section{Our model and assumptions}\label{S-model}
Let $\varphi:(0, \infty)\to (0, \infty)$ be a locally bounded function satisfying a {\it weak upper
scaling property} with exponent $\beta\in(0, 2)$ i.e., 
\begin{equation}\label{EA1}
\varphi(st) \leq \kappa_\circ s^{\beta} \varphi(t) \quad \text{for} \quad s\geq 1, t > 0\,,\tag{A1}
\end{equation}
for some $\kappa_\circ>0$. We also assume that 
\begin{equation}\label{EA2}
\int_{0}^{1} \frac{\varphi(y)}{y} \D{y} \,<\, \infty\,.\tag{A2}
\end{equation}
Note that \eqref{EA1} and \eqref{EA2} give us 
$$\int_{\Rd} (\abs{y}^2\wedge 1)\frac{\varphi(1/|y|)}{|y|^d}\, \D{y}\,<\, \infty.$$
The ellipticity class is defined with respect to the set of nonlocal operators $\sL$ containing operator 
$L$ of the form
\begin{equation}\label{E1.1}
L u(x) =\,\int_{\Rd}\bigl(u(x+y)+u(x-y)-2u(x)\bigr) \frac{k(y)}{\abs{y}^{d}}\, \D{y}\,,
\end{equation}
where for some fixed $\lambda, \Lambda, 0<\lambda\leq\Lambda$, it holds that $k(y)=k(-y)$ and
\begin{equation}\label{EA3}
(2- \alpha) \lambda \frac{1}{{\abs y}^{\alpha}} 
\;\leq\; k(y) \leq  \Lambda \left( \frac{2- \alpha}{{\abs y}^{\alpha}} + 
\varphi({1}/{\abs y}) \right) \quad \text{for some}\; \; \alpha\in (\beta, 2).\tag{A3}
\end{equation}
Then the extremal Pucci operators (with respect to $\sL$) are defined to be $\cM^+ u =\sup_{L\in\sL} Lu$ and $\cM^- u =\inf_{L\in\sL} Lu$. Defining $\delta (u, x, y)= u(x+y) + u(x-y) - 2u(x)$, we find 
from \eqref{EA3} that
\begin{align*}
\cM^{+}u(x) &= \int_{\Rd} \frac{\Lambda \delta^{+}(u,x,y)}{\abs y ^{d}} \left( \frac{ 2- \alpha }{\abs y ^{\alpha}} + \varphi(1/{\abs y}) \right ) - \frac{\lambda \delta^{-}(u,x,y)}{\abs y ^{d}} \left( \frac{ 2- \alpha }{\abs y ^{\alpha}} \right) \D{y},
\\
\cM^{-}u(x) &= \int_{\Rd} \frac{\lambda \delta^{+}(u,x,y)}{\abs y ^{d}} \left( \frac{2- \alpha}{\abs y ^{\alpha}} \right ) - \frac{\Lambda \delta^{-}(u,x,y)}{\abs y ^{d}} 
\left( \frac{ 2- \alpha}{\abs y ^{\alpha}} + \varphi(1/{\abs y}) \right ) \D{y}\,.
\end{align*}
In this article we would be interested in operators that are elliptic with respect to the class
$\sL$ (see \cite[Definition~3.1]{CS09}). Recall that a nonlinear operator $I$ is said to be elliptic
with respect to the class $\sL$ if it holds that
$$\cM^{-}(u-v)\;\leq\; Iu-Iv\;\leq\;\cM^+(u-v)\,.$$
It should be observed that if a operator is elliptic with respect to a subset of $\sL$ it is also elliptic 
with respect to $\sL$. For instance, if we let $\varphi(r)=r^\beta, \beta\in (0, \alpha)$, and $\sL_1\subset\sL$ be the collection of all kernel functions $k$ satisfying
$$(2- \alpha) \lambda \frac{1}{{\abs y}^{\alpha}} 
\;\leq\; k(y) \leq  \Lambda \left( \frac{2- \alpha}{{\abs y}^{\alpha}} + 
\Ind_{\sB^c_1}(y) \frac{1}{|y|^\beta} \right),$$
then the results of this article hold for any operator that is elliptic with respect to $\sL_1$. We remark
that this class of operators are not covered by \cite{KKL16,KL20}.

Also, by  sub-solutions and super-solutions we shall mean
viscosity sub and super-solutions, respectively. 
\begin{definition}
A bounded function $u:\Rd\to \RR$ which is upper-semicontinuous (lower-semicontinuous) in $\bar\Omega$ 
is said to be a viscosity subsolution (supersolution) of $I u =f$ in $\Omega$ and written as 
$Iu\geq f$ ($Iu\leq f$) in $\Omega$, if the following holds: if a $\cC^2$ function $\psi$ touches
$u$ at $x\in\Omega$ from above (below) in a small neighbourhood $N_x\Subset\Omega$, i.e., $\psi\geq u$
in $N_x$ and $\psi(x)=u(x)$, then the function $v$ defined by
\[
v(y)=\left\{\begin{array}{ll}
\psi(y) & \text{for}\; y\in\,N_x,
\\
u(y) & \text{otherwise}\,,
\end{array}
\right.
\]
satisfies $Iv(x)\geq f(x)$ ($Iv(x)\leq f(x)$, resp.). A function $u$ is said to be a viscosity solution
if $u$ is both a viscosity subsolution and a viscosity supersolution.
\end{definition}
We refer to \cite{BCC,BC08,CS09} for more details on viscosity solutions. We also remark that the boundedness assumption of $u$ assures integrability of $u$ at infinity  with respect to the jump kernel. This can be removed by
assuming suitable integrability criterion and the results of this article will remain valid.

We also need scaled extremal operators which we introduce now.
Define $\varphi_{i}(|y|) = \frac{\kappa_\circ}{(2^{i})^{(\alpha - \beta)}} \varphi({\abs y})$ for 
$i\geq 0$.
The scaled extremal Pucci operators are defined to be
\begin{align}
\cM_{i}^{+}u(x) &=  \int_{\Rd} \frac{\Lambda \delta^{+}(u,x,y)}{\abs y ^{d}} 
\left( \frac{2- \alpha }{\abs y ^{\alpha}} + \varphi_{i} (1/{\abs y}) \right ) - \frac{\lambda \delta^{-}(u,x,y)}{\abs y ^{d}} \left( \frac{ 2- \alpha}{\abs y ^{\alpha}} \right) \D{y}\,,\label{E2.3}
\\
\cM_{i}^{-} u(x) &= \int_{\Rd} \frac{\lambda \delta^{+}(u,x,y)}{\abs y ^{d}} 
\left( \frac{2- \alpha}{\abs y ^{\alpha}} \right ) - \frac{\Lambda \delta^{-}(u,x,y)}{\abs y ^{d}} \left( \frac{ 2- \alpha }{\abs y ^{\alpha}} + \varphi_{i} (1/{\abs y}) \right ) \D{y}\,. \label{E2.4}
\end{align}

We conclude this section with a brief motivation for the above model. The works of Caffarelli
and Silvestre \cite{CS09,CS11,CS11a} are based on nonlinear generalization of the classical fractional Laplacian.
These Pucci type operators appear in control theory when the underlying controlled dynamics is
governed by stable-like process. Of course, one can consider nonlocal Pucci operators corresponding
to other L\'{e}vy processes. For instance, \cite{KL20} considers a class of nonlinear integro-differential
operator corresponding to a family of subordinate Brownian motion. Subordinate Brownian motion(sBM) forms
an important family of L\'evy process. For a large class of
sBM, the jump kernel (or density of L\'evy measure) is proportional to $|y|^{-d}\Phi(1/|y|^2)$ where
$\Phi$ is a Bernstein function, in particular, increasing and concave. For more details we refer \cite{SSV}.
Note that $\Phi(r)=r^{\alpha/2}$ corresponds to the stable kernel. The recent work \cite{KL20} deals
with kernel of the form $|y|^{-d}\Phi(1/|y|^2)$ where $\Phi$ has both lower and upper weak scaling propery.
Our present model corresponds to a L\'evy process which is obtained by adding two independent L\'evy
process: one is $\alpha$-stable process and the other one generated by the L\'evy measure
$|y|^{-d}\varphi(1/|y|)\D{y}$. Since $\varphi$ need not have a lower weak scaling property, the present 
model is not covered by \cite{KL20}. Furthermore, for the results of \cref{S-ABP} and  H\"{o}lder regularity
to hold we only allow
$\varphi$ in the upper bound of the kernel functions $k$  (and not necessarily in the lower bound), as mentioned in \eqref{EA3}. So the above operator allows  \textit{degeneracy} in lower order operator.

%%%%%%%%%%%%%%%%%%%%%%%%%%%%%%%%%%%%%%%%%%%%%%%%%%%%%%%%%%%%%%%%%%%%%%%%%%%%%%%%%%%%%%%%%%%%%%%%%%%%%
\section{ABP estimates and weak Harnack inequality}\label{S-ABP}
In this section we obtain an Aleksandrov-Bakelman-Pucci (ABP) estimate which is the main ingredient
for weak-Harnack inequality and point estimate. Let us begin by defining the concave envelope and
contact set. Let $u$ be a function that is non-positive outside $\sB_1$ (unit ball around $0$).
The concave envelope $\Gamma$ of $u$ in $\sB_3$ is defined as follows
\[
\Gamma(x)=\left\{
\begin{array}{ll}
\inf\{p(x)\; :\; p \;  \text{is a plane satisfying}\;  p\geq u^+\; \text{in}\; \sB_3\}& \text{in}\; \sB_3,
\\
0 & \text{in}\; \sB_3^c.
\end{array}
\right.
\]
The contact set is defined to be $\Sigma=\{\Gamma=u\}\cap \sB_1$.
Here and in what follows we use the notation $\sB_r$ to denote the ball of radius $r$ around $0$, and
by $\sB_r(x)$ we would denote the ball of radius $r$ around $x$.
The following lemma
follows by adapting \cite[Lemma~8.1]{CS09} in our setting.

\begin{lemma}\label{L3.1}
Let $u \leq 0$ in $\Rd \setminus \sB_1$ and $\Gamma$ be its concave envelope in $\sB_3$. Assume $\cM_{i}^{+}u(x) \geq -f(x)$ in $\sB_1$ for some $i \geq 0$. Let $\rho_0 = \nicefrac{1}{16 \sqrt{d}}$,  $r_k = \rho_0 2^{-\nicefrac{1}{2-\alpha}} 2^{-k}$, and 
$R_k(x) = \sB_{r_k}(x) \setminus \sB_{r_{k+1}}(x)$. Then there exists a constant $C_0$ independent of $i \geq 0$ and $\alpha$ such that for any $x \in \Sigma$ and any $M > 0$ there is a $k$ satisfying 
$$
\vert R_k(x) \cap \{ u(y) < u(x) +(y-x) \cdot\nabla \Gamma(x) - Mr_k^2 \} \vert \leq C_0 \frac{f(x)}{M} \vert R_k(x) \vert.
$$
Furthermore, $C_0$ depends only on $(\lambda , d  , \rho_0)$.
\end{lemma}

\begin{proof}
First we notice that $\cM_{0}^{+}u \geq M_{i}^{+}u$ for all $i$. Therefore 
$\cM_{i}^{+}u(x) \geq -f(x)$ implies that $\cM_{0}^{+}u(x) \geq -f(x)$.
Hence it is enough to prove lemma for the case $i=0$.

Let $x \in \Sigma$ and recall that $\delta(u,x,y)=u(x+y) + u(x-y) - 2u(x)$.
 If both $x+y \in \sB_3$ and $x-y \in \sB_3$ then 
$\delta(u,x,y) \leq 0$ since $u(x)= \Gamma(x)=p(x)$ for some plane that remains above $u$ in $\sB_3$. If either $x+y$ or $x-y$ $\in \sB_3^c$ then both $x+y$ and $x-y$ $\in \sB_1^c$ and since $u(x) = \Gamma(x) \geq 0 $ we have $\delta(u,x,y) \leq 0$. Thus, using \eqref{E2.3}, we find

\begin{align}\label{EL3.1A}
-f(x) \leq \cM_{0}^{+}u(x)  &= (2- \alpha) \int_{\Rd}  \frac{-\lambda \delta^{-}(u,x,y)}{\abs y ^{d}} \left( \frac{1}{\abs y ^{\alpha}} \right) \D{y}\nonumber
\\
&\leq (2- \alpha) \int_{\sB_{r_0}}  \frac{-\lambda \delta^{-}(u,x,y)}{\abs y ^{d}}
\left( \frac{1}{\abs y ^{\alpha}} \right) \D{y}\,,
\end{align}
where $r_0 = \rho_0 2^{-\frac{1}{2-\alpha}}$. Let
$$E_k^{\pm}\df\{ R_k \cap \{ u(x \pm y) < u(x) \pm y \cdot\nabla \Gamma(x) - M r_k^2 \} \} .$$
Then on this set we will have $\delta^{-}(u,x,y) \geq 2M r_k^2$. Also $ \vert E_k^{+} \vert = \vert E_k(x) \vert$ where $ E_k(x) \df \{R_k(x) \cap \{ u(y) < u(x) +(y-x) \cdot\nabla \Gamma(x) - Mr_k^2 \} \}$. 
Now suppose that the result does not hold for any $C_0$. We will arrive at contradiction for large enough $C_0$. Using \eqref{EL3.1A} we obtain that
\begin{align*}
f(x) \geq (2- \alpha) \lambda \sum_{k=0}^{\infty} \int_{R_k} \frac{\delta^{-}(u,x,y)}{\abs y ^{d+\alpha}} \D{y}
& \geq (2- \alpha) \lambda \sum_{k=0}^{\infty} \int_{E_k} \frac{2M r_k^2}{\abs y ^{d+\alpha}} \D{y}
\\
& \geq  2 (2- \alpha) \lambda \sum_{k=0}^{\infty} M \frac{r_k^2}{r_k^{d+\alpha}} \vert E_k \vert 
\\
& \geq 2 (2- \alpha) \lambda \sum_{k=0}^{\infty} M \frac{r_k^2}{r_k^{d+\alpha}} \frac{C_0 f(x)}{M} \vert R_k \vert 
\\
&\geq 2 (2- \alpha) \lambda \left[ \sum_{k=0}^{\infty}  \frac{r_k^2}{r_k^{d+\alpha}} \omega_d (r_k^d - r_{k+1}^d)    \right] C_0 f(x)
 \\
&= 2 (2- \alpha) \lambda \omega_d \left[ \sum_{k=0}^{\infty}  r_k^{2 - \alpha} \left( 1 - \left(\frac{1}{2}\right)^d\right)    \right] C_0 f(x),
\end{align*}
since $\frac{r_{k-1}}{r_k} = \frac{1}{2}$ for any $k$, where $\omega_d$ denotes volume of the unit ball.
Now we notice that $\sum_{k=0}^{\infty}  r_k^{2 - \alpha} $ is a geometric series, and therefore,
$$
f(x) \geq 2 (2- \alpha) \lambda \omega_d  \left( 1 - \left(\frac{1}{2}\right)^d \right) \left[ \frac{\rho_0^{2- \alpha}}{2} \left( \frac{1}{1-2^{-(2-\alpha)}} \right) \right] C_0 f(x)\,.
$$
Take 
$$c=  \lambda \omega_d  \left( 1 - \left(\frac{1}{2}\right)^d \right)  \left( \frac{\rho_0^{2} (2- \alpha) }{1-2^{-(2-\alpha)}} \right),$$
and since $ \frac{ (2- \alpha) }{1-2^{-(2-\alpha)}} $ remains bounded
below for all $\alpha \in (0,2)$, we have $c$ positive for any $\alpha \in (0,2)$. Thus
$$
f(x) \geq c\,  C_0 f(x).
$$
Choosing $C_0>c^{-1}$ leads to a contradiction. Hence the proof.
\end{proof}

Using Lemma~\ref{L3.1} and the arguments in \cite[Section~8]{CS09} we arrive at the following 
result. This is a mild extension to \cite[Theorem~8.7]{CS09}.

\begin{theorem}\label{T3.1}
Let $u$ and $\Gamma$ be same as in Lemma~\ref{L3.1}. Then there is a finite family of open cubes $\cQ_j$ with diameters $d_j$ such that following hold.
\begin{enumerate}
\item[(i)] Any two cubes $\cQ_i$ and $\cQ_j$ in the family do not intersect.
\item[(ii)] $\{ u= \Gamma  \} \subset \bigcup_{j=1}^m \cQ_j$
\item[(iii)] $\{ u= \Gamma  \} \cap \overline{\cQ}_j \neq \emptyset$ for any $\cQ_j$.
\item[(iv)] $d_j \leq \rho_0 2^{- \nicefrac{1}{2-\alpha}}$, where $\rho_0 = \nicefrac{1}{16 \sqrt{d}}$.
\item[(v)] $ \vert \nabla \Gamma(\overline{\cQ}_j) \vert \leq C (\max_{\overline{\cQ}_j}f(x))^d \vert \overline{\cQ}_j \vert $.
\item[(vi)] $\vert \{ y \in 8\sqrt{d} \cQ_j : u(y) > \Gamma(y) - C (\max_{\overline{\cQ}_j}f(x)) d_j^2\} \vert \geq \mu \vert \cQ_j \vert$.
\end{enumerate}
The constant $C>0$ and $\mu > 0$ depends only on $(\lambda  , d  , \rho_0)$ but not on $i$ and 
$\alpha$.

\end{theorem}

Next we consider a \textit{special function} which will play a key role in our analysis on
point estimate and weak-Harnack inequality. Let $p>0$ and $\delta$ be small positive number.  Define
$$
f(x) \df \min \{ \delta^{-p}, \max \{ \abs x^{-p} , (2\sqrt{n})^{-p}\}\}\,.
$$ 
We claim that, for a given $r\in (0, 1)$, we can choose $p$ and $\delta$  so that
\begin{equation}\label{E3.2}
\cM_0^{-} f(x)\geq 0\quad \text{for}\;  r<\abs{x}\leq 2\sqrt{n}\,.
\end{equation}
For any $0<r<1$, define
$$
\hat{f}(x) = \min \left\{ \left(\frac{\delta}{r}\right)^{-p}, \max \left\{ \abs x^{-p} , \left(\frac{2\sqrt{n}}{r} \right)^{-p}\right\}\right\}\,.
$$
Then clearly, $f(rx)= r^{-p} \hat{f}(x)$ and for any $\abs x \geq r$ we have
\begin{align*}
\int_{\Rd} \left(f(x+y)+f(x-y)-2f(x)\right) \frac{k(y)}{\abs y^d} \D{y} 
&= \int_{\Rd} \left(f(x+ry)+f(x-ry)-2f(x)\right) \frac{k(ry)}{\abs y^d} \D{y} 
\\
&= r^{-p} \int_{\Rd} \delta(\hat{f}, \nicefrac{x}{r}, y) \frac{k(ry)}{\abs y^d} \D{y}\,.
\end{align*}
Therefore, to establish \eqref{E3.2}  it is enough tho show that for all 
$1\leq \abs x \leq \frac{2 \sqrt{n}}{r}$,
\begin{equation}\label{E3.3}
r^{-p} \inf_{k} \int_{\Rd} \delta(\hat{f}, x, y) \frac{k(ry)}{ \abs y^d} \D{y} \geq 0\,,
\end{equation}
where infimum is taken over all kernel $k$ satisfying \eqref{EA3}.
Note that $\hat{f}$ is radially non-increasing function. Fix $\abs{x} \geq 1$ and define $ \tilde{f}(y)= \abs{x}^p \hat{f}(\abs x y)$. Then it implies that $\tilde{f}(y) \geq \hat{f}(y)$,  for
all $y\in \Rd$ and 
$\tilde{f}(x/{\abs x}) = \hat{f}(x)$. Thus we obtain
\\
$$
\delta(\hat{f}, x, y) = \frac{1}{\abs x^p} \left[ \tilde{f}(\frac{x+y}{\abs x}) + \tilde{f}(\frac{x-y}{\abs x}) - 2 \tilde{f}(\frac{x}{\abs x}) \right] \geq \frac{1}{\abs x^p} \left[ \hat{f}(\frac{x +y}{\abs x}) + \hat{f}(\frac{x-y}{\abs x}) - 2\hat{f}(\frac{x}{\abs x})\right].
$$
Without any loss of generality we may assume that $x/|x|=e_1=(1,\ldots, 0)$.
Then
\begin{align*}
\int_{\Rd} \delta(\hat{f}, x, y) \frac{k(ry)}{ \abs y^d} \D{y} 
&\geq  \frac{1}{\abs x^p} \int_{\Rd} \left[ \hat{f}(\frac{x +y}{\abs x}) + \hat{f}(\frac{x-y}{\abs x}) - 2\hat{f}(\frac{x}{\abs x})\right] \frac{k(ry)}{ \abs y^d} \D{y}
\\
&= \frac{1}{\abs x^p} \int_{\Rd} \left[ \hat{f}(\frac{x}{\abs x}+y) + \hat{f}(\frac{x}{\abs x} - y) - 2\hat{f}(\frac{x}{\abs x})\right] \frac{k(r \abs x y)}{ \abs y^d} \D{y} 
\\
&= \frac{1}{\abs x^p} \int_{\Rd} \left[ \hat{f}( e_1 +y) + \hat{f}( e_1 - y) - 2\hat{f}(e_1)\right] \frac{k(r \abs x y)}{ \abs y^d} \D{y} 
\\
&\geq \frac{1}{\abs x^p} \int_{\Rd} \delta(\hat{f}, e_1, y) \frac{k(r \abs x y)}{ \abs y^d} \D{y}\,.
\end{align*}
Hence, by \eqref{E2.4}, we get
\begin{align*}
\inf_{k} \int_{\Rd} \delta(\hat{f}, x, y) \frac{k(ry)}{ \abs y^d} \D{y} 
& \geq\, \frac{1}{ \abs x^p} \int_{\Rd} \frac{\lambda \delta^{+}(\hat{f}, e_1, y)}{\abs y ^{d}}
\left( \frac{2- \alpha}{ (r \abs x)^{\alpha}  \abs{y}^{\alpha}} \right) 
\\
&\qquad - \frac{\Lambda \delta^{-}(\hat{f}, e_1, y)}{\abs{y}^{d}} \left(\frac{2- \alpha}{(r \abs x)^{\alpha} 
\abs{y}^{\alpha}} + \varphi_0(1/{ r \abs x \abs y}) \right) \D{y}
\\
&\df I_1 - I_2\,.
\end{align*}
We now recall the following elementary relations that hold for any $a > b > 0$ and $q > 0$:
$$
(a+b)^{-q} \geq a^{-q} \left( 1 - q \frac{b}{a} \right),
$$
$$
(a+b)^{-q} + (a-b)^{-q} \geq 2a^{-q} + q(q+1)b^2 a^{-q-2}\,.
$$
Fixing $\delta<\frac{r}{2}$,  we then see that for $|y|< 1/2$,
\begin{align}\label{E3.4}
\delta(\hat{f}, e_1, y) &= \vert e_1+y \vert^{-p} + \vert e_1-y \vert^{-p} -2\nonumber
\\
&= (1+ \abs y^2 + 2y_1)^{\nicefrac{-p}{2}} +(1+ \abs y^2 - 2y_1)^{\nicefrac{-p}{2}} -2\nonumber
\\
&\geq 2(1+\abs y^2)^{\nicefrac{-p}{2}} + p(p+2) y_1^2 (1+ \abs y^2)^{\nicefrac{-p}{2} -2} -2\nonumber
\\
&\geq p \left( -\abs y^2  +(p+2) y_1^2 - \frac{1}{2} (p+2)(p+4) y_1^2 \abs{y}^2\right)\,.
\end{align}
Let us first calculate $I_2$. For any $\abs{y} < \frac{1}{2}$ we have from \eqref{E3.4} that 
$$
\delta^{-}(\hat{f}, e_1, y) \leq p \left( 1  + \frac{1}{2} (p+2)(p+4)\right) \abs{y}^2\,.
$$
Denote by $C_p = p \left( 1 + \frac{1}{2}  (p+2)(p+4) \right)$. Then
\begin{align*}
I_2 & \leq C_p \Lambda \int_{\abs y < \frac{1}{2}} \frac{\abs y^2}{\abs y ^{d}} 
\left[\frac{ 2- \alpha}{ (r \abs x)^{\alpha} \abs{y}^{\alpha}} + \varphi_0(1/{ r \abs x \abs y})\right] \D{y}
\\
&\qquad + \Lambda \int_{\abs y \geq \frac{1}{2}} \frac{2 \hat{f}(e_1)}{\abs y ^{d}} 
\left[ \frac{ 2- \alpha}
 { (r \abs x)^{\alpha} \abs{y}^{\alpha}} + \varphi_0(1/{ r \abs x \abs y}) \right] \D{y}
\\
&=I_{21} + I_{22}\,.
\end{align*}
We observe that
\begin{align*}
 \frac{ 2- \alpha}{ (r \abs x)^{\alpha}} \int_{\abs y < \frac{1}{2}} \frac{\abs y^2}{\abs y ^{d}}  \frac{1}{\abs{y}^{\alpha}} \D{y} &=  \frac{\omega_d}{(r \abs x)^{\alpha}} 
 \left( \frac{1}{2} \right)^{2-\alpha}\,,
 \\ 
\frac{2- \alpha}{ (r \abs x)^{\alpha}} \int_{\abs y \geq \frac{1}{2}} \frac{2}{\abs y ^{d}}
 \frac{1}{ \abs y ^{\alpha}} \D{y} &=  \frac{\omega_d}{ (r \abs x)^{\alpha}} 
 \frac{2- \alpha}{\alpha} 2^{\alpha + 1}\,.
\end{align*}
On the other hand, for $r \abs x \leq 2 \sqrt{n}$,
\begin{align*}
\int_{\abs y < \frac{1}{2}} \frac{\abs y^2}{\abs y ^{d}} 
\varphi_0(1/{ r \abs x \abs y}) \D{y}  
&\leq \kappa_\circ \left(\frac{2 \sqrt{n}}{r \abs x}\right)^{\beta} \int_{\abs y < \frac{1}{2}} \frac{\abs y^2}{\abs y ^{d}} \frac{1}{\abs y^{\beta}}
\varphi(\frac{1}{ 2 \sqrt{n}})\, \D{y}
\\
&= \kappa_\circ \left(\frac{2 \sqrt{n}}{r \abs x}\right)^{\beta} 
\varphi(\frac{1}{2 \sqrt{n}}) \frac{\omega_d}{2-\beta} \left(\frac{1}{2} \right)^{2-\beta} \,,
\end{align*}
using the fact $ \abs y < \frac{1}{2}$ and \eqref{EA1}. Again using \eqref{EA1}-\eqref{EA2}
\begin{align*}
\int_{\abs y \geq \frac{1}{2}} \frac{2}{\abs y ^{d}} \varphi_0(\frac{1}{ r \abs x \abs y}) \D{y} 
&= 2 \omega_d \int_{\nicefrac{1}{2}}^{\infty} \frac{1}{t} \varphi(\frac{1}{r\abs x t}) \D{t}
 \\
 &\leq 2 \kappa_\circ \omega_d \left(\frac{2 \sqrt{n}}{r \abs x}\right)^{\beta} 
 \int_{\nicefrac{1}{2}}^{\infty} \frac{1}{t} \varphi(\frac{1}{2 \sqrt{n} t}) \D{t} 
 \\
&= 2 \kappa_\circ \omega_d \left(\frac{2 \sqrt{n}}{r \abs x}\right)^{\beta} 
\int_{\nicefrac{2 \sqrt{n}}{2}}^{\infty} \frac{1}{t} \varphi(\frac{1}{t}) \D{t}
\\
&= 2 \kappa_\circ \omega_d \left(\frac{2 \sqrt{n}}{r \abs x}\right)^{\beta} 
\int_{0}^{\nicefrac{1}{ \sqrt{n}}} \frac{1}{t} \varphi(t) \D{t}
\\
&\leq 2 \kappa_\circ \omega_d \left(\frac{2 \sqrt{n}}{r \abs x}\right)^{\beta} \kappa_1,
\end{align*}
for some constant $\kappa_1$ depending only on $\varphi$.  Thus combining we obtain
for $1\leq |x|\leq \frac{2\sqrt{n}}{r}$
\begin{align}\label{E3.5}
I_2 \leq  C_p \Lambda \left[  \frac{\omega_d}{ r^{\alpha}} \left( \frac{1}{2} \right)^{2-\alpha} \right] &+ \Lambda \left[ \frac{\omega_d}{ r^{\alpha}} \frac{2 - \alpha}{\alpha} 2^{\alpha + 1} \right] 
 + C_p \kappa_\circ \Lambda \left[ \left(\frac{2 \sqrt{n}}{r }\right)^{\beta} \varphi(\frac{1}{2 \sqrt{n}}) \frac{\omega_d}{2-\beta} \left(\frac{1}{2} \right)^{2-\beta} \right] \nonumber
 \\
& + \kappa_\circ \Lambda \left[ 2 \omega_d \left(\frac{2 \sqrt{n}}{r}\right)^{\beta} \kappa_1 \right]\,.
\end{align}
Next we calculate $I_1$.  Notice that if $\delta < \frac{r}{16}$, 
then $\delta(\hat{f}, e_1, y) \geq (\delta/r)^{-p}$ for all $y \in \sB_{\nicefrac{\delta}{4r}}(e_1)$. Hence
$$
I_1 = \int_{\Rd} \frac{\lambda \delta^{+}(\hat{f}, e_1, y)}{\abs y ^{d}}\, \frac{ 2-\alpha}{(r \abs x)^{\alpha}  \abs y ^{\alpha}}  \; \D{y} \geq  \frac{\lambda(2-\alpha)}{(r \abs x)^{\alpha}} \kappa_2 
(\delta/r)^{d-p}\geq \kappa_2 \frac{\lambda(2-\alpha)}{(2\sqrt{n})^{\alpha}}  
(\delta/r)^{d-p}\,,
$$
for some constant $\kappa_2$. Thus choosing $p>d$ and $\delta$ small enough \eqref{E3.3} follows
from \eqref{E3.5}. This leads to the following

\begin{lemma}\label{L3.2}
Given any $r, n>0$ there are positive $p$ and $\delta$ such that the function
$$
f(x) = \min \{ \delta^{-p}, \max \{ \abs{y}^{-p} , (2\sqrt{n})^{-p}\}\}\,,
$$
is a solution to
$$
\cM_0^{-}f(x) \geq 0,
$$
for every $0<\alpha_0 \leq \alpha \leq \alpha_1 < 2$ and $\abs{x} >r$.
\end{lemma}

\begin{proof}
For any $\abs{x}\geq 2\sqrt{n}$, by the definition of $f$, we get that 
$\delta(f,x,y) = f(x+y) + f(x-y) - 2f(x) \geq 0 $ for all $y \in \Rd$.
Therefore, $\cM_0^{-}f(x) \geq 0$ for all $\abs x \geq 2\sqrt{n}$.
Hence the proof follows from \eqref{E3.2}.
\end{proof}
Applying Lemma~\ref{L3.2} we obtain the following corollary. The proof would be same as \cite[Corollary~9.3]{CS09} and thus omitted.
\begin{corollary}\label{C3.1}
Given any $\alpha \in [ \alpha_0,\alpha_1]$ and $i \geq 0$ there is a function $\Phi$ such that
\begin{enumerate}
\item[(i)] $\Phi$ is continuous in $\Rd$ ,
\item[(ii)] $\Phi(x) = 0$ for x outside $\sB_{2\sqrt{n}}$,
\item[(iii)] $\Phi > 2$ for $ x \in \cQ_3$, and
\item[(iv)] $\cM_i^{-}\Phi > \psi(x) $ in $\Rd$ for some positive function $\psi$ supported in 
$\overline{\sB}_{\nicefrac{1}{4}}$.
\end{enumerate}
\end{corollary}

Using \cref{T3.1,C3.1} and repeating the arguments of \cite[Lemma~10.1]{CS09} we arrive at the
following result.

\begin{lemma}\label{L3.3}
There exist constants $ \varepsilon_0 > 0$ , $0 < \mu < 1$, and $M > 1$ 
(depending only on $d,\lambda,\Lambda,\alpha,\varphi$), such that if 
\begin{enumerate}
\item[(i)] $u \geq 0 $ in $\Rd$,
\item[(ii)] $\inf_{\cQ_3} u \leq 1$, and 
\item[(iii)] $\cM_i^{-} u \leq \varepsilon_0$ in $\cQ_{4\sqrt{d}}$ for some $i \geq 0$,
\end{enumerate}  
then $\vert \{u \leq M  \} \cap \cQ_1  \vert > \mu$.
\end{lemma}

The above lemma is a key tool in obtaining weak-Harnack estimate. Combining a Calder\'{o}-Zygmund 
type argument with \cref{L3.3} we obtain the following

\begin{theorem}\label{T3.2}
There exist constant $ \tilde\varepsilon > 0$ , $0 < \tilde\mu < 1$, and $\tilde{M} > 1$ 
(depending only on $d,\lambda,\Lambda,\alpha,\varphi$), such that if 
\begin{enumerate}
\item[(i)] $u \geq 0 $ in $\Rd$
\item[(ii)] $\inf_{\cQ_3} u \leq 1$, and 
\item[(iii)] $\cM^{-} u \leq \tilde\varepsilon$ in $\cQ_{4\sqrt{d}}$,
\end{enumerate}
then
$$
\vert \{ u \geq \tilde{M}^k \} \cap \cQ_1  \vert \leq (1 - \tilde\mu)^k
$$
for $k \in \mathbb{N}$. As a consequence, we have that
$$
\vert \{ u \geq t \} \cap \cQ_1  \vert \leq \kappa\, t^{-\varepsilon} \quad \forall\; t\,>\,0\,,
$$
for some universal constant $\kappa, \varepsilon$.
\end{theorem}

\begin{proof}
The case $k=1$ follows from \cref{L3.3}. Note that $\kappa_\circ=1$ implies
$\cM^-_0=\cM^-$ and thus from the arguments of \cref{L3.3} we can obtain the constants
$(\varepsilon_1, M_1, \mu_1)$ satisfying \cref{L3.3} with operator $\cM^-$.
We set $\tilde\varepsilon=\varepsilon_0\wedge\varepsilon_1$, $\tilde{M}=M\vee M_1$ and
$\tilde{\mu}=\mu\wedge\mu_1$.
Now using induction hypothesis assume that theorem is true for $m=1, \ldots, k-1$, and denote by 
$$
A = \{ u > \tilde{M}^k\} \cap \cQ_1 , \quad \quad B= \{ u > \tilde{M}^{k-1}\} \cap \cQ_1\,.
$$
Thus we only need to show that
\begin{equation}\label{ET3.2A}
\abs{A} \leq\, (1 - \tilde\mu)\abs{B}\,.
\end{equation}
Clearly, $A \subset B \subset \cQ_1$ and 
$\abs{A} \leq \vert \{ u > \tilde{M}\} \cap \cQ_1 \vert  \leq 1- \tilde\mu$. 
We show that if $\cQ \df \cQ_{\nicefrac{1}{2^i}}(x_0) $  is a dyadic cube such that
\begin{equation}\label{ET3.2B}
\vert A \cap \cQ \vert > (1-\tilde\mu)|\cQ|,
\end{equation}
then $\widetilde{\cQ} \subset B$. where $\widetilde{\cQ} $ is a predecessor of $\cQ$.
Then \eqref{ET3.2A} follows from \cite[Lemma~4.2]{CC95}. Suppose that 
$\widetilde{\cQ} \nsubseteq B$. Take 
$$
\tilde{x} \in \widetilde{\cQ} \quad \text{such that} \quad u(\tilde{x}) \leq \tilde{M}^{k-1}\,.
$$
We consider the function
$$
v(y) \df \frac{u(x_0 + \frac{1}{2^i} y)}{\tilde{M}^{k-1}}\,.
$$
Clearly, $v \geq 0 $ in $\Rd$ and $\inf_{\cQ_3} v \leq 1$. We claim that
$\cM_{i}^{-} v \leq \varepsilon_0$ in $\cQ_{4 \sqrt{d}}$ where $\varepsilon_0$ is given by \cref{L3.3}.
Let $\hat{x} = x_0 + \frac{1}{2^i} x$ for some $x \in \cQ_{4 \sqrt{n}}$. Then simple calculation shows that
$$
\frac{1}{\tilde{M}^{k-1}} \delta (u , \hat{x}, \frac{y}{2^i}) = \delta(v , x , y)\,,
$$
and using \eqref{EA1}, we obtain
\begin{align*}
\frac{1}{ (2^i)^{\alpha} \tilde{M}^{k-1}} \cM^{-}u(\hat{x}) 
&= \frac{1}{ (2^i)^{\alpha} \tilde{M}^{k-1}} \int_{\Rd} \Bigl[\frac{\lambda 
\delta^{+}(u,\hat{x},\frac{y}{2^i})}{\abs y ^{d}} \left( \frac{(2 - \alpha) (2^i)^{\alpha}  }{\abs y ^{\alpha}} \right )
\\
&\hspace{3em} - \frac{\Lambda \delta^{-}(u,\hat{x},\frac{y}{2^i})}{\abs y ^{d}} 
\left( \frac{(2 - \alpha) (2^i)^{\alpha}}{\abs y ^{\alpha}} + \varphi({2^i}/{ \abs y}) \right )\Bigr] \D{y}
 \\
 &\geq \frac{1}{\tilde M^{k-1}} \int_{\Rd} \frac{\lambda \delta^{+}(u,\hat{x},\frac{y}{2^i})}{\abs y ^{d}} \left( \frac{2 - \alpha }{\abs y ^{\alpha}} \right )
 \\
 &\hspace{3em} - \frac{\Lambda \delta^{-}(u,\hat{x},\frac{y}{2^i})}
 {\abs y ^{d}} \left( \frac{ 2 - \alpha}{\abs y ^{\alpha}} + \frac{\kappa_\circ}{2^{i (\alpha - \beta)}} \varphi(\frac{1}{ \abs y}) \right ) \D{y}
 \\
 &\geq   \cM_{i}^{-} v(x)\,.
\end{align*}
Thus we have the claim $\cM_{i}^{-} v \leq \varepsilon_0$ in $\cQ_{4 \sqrt{d}}$. Therefore, we can apply
\cref{L3.3} to obtain
$$\tilde\mu< \vert \{ v(x)  \leq M \} \cap \cQ_1 \vert = 2^{id} \vert \{ u(x)  \leq M^{k} \} \cap \cQ \vert\,$$
implying
$$\vert \{ u(x)  \leq M^{k} \} \cap \cQ \vert\,>\,\tilde\mu |\cQ|.$$
This gives us \eqref{ET3.2B}. This completes the proof.
\end{proof}

\begin{remark}
Note that the constants $(\tilde\varepsilon, \tilde{M},  \tilde\mu)$ in \cref{T3.2} also
work if we replace $\cM^{-}$ by $\cM_i^-$ for all $i\geq 0$.
\end{remark}

By a standard covering argument we obtain following result
\begin{theorem}\label{T3.3}
Let $ u \geq 0$ in $\Rd$ , $u(0) \leq 1$, and $\cM_i^{-} u \leq \tilde\varepsilon_0$ in $\sB_2$. Then
$$
\vert  \{ u \geq t \} \cap \sB_1   \vert \leq C\, t^{-\varepsilon} \:   \: \text{for  every} \; t>0\,,
$$
where the constant $C$ and $\varepsilon$ depend only on ($d,\lambda,\Lambda,\alpha,\varphi$).
\end{theorem}

We conclude the section by proving a weak-Harnack estimate.
\begin{theorem}\label{T3.4}
Let $ u \geq 0$ in $\Rd$ and $\cM^{-} u \leq C_0$ in $B_{2r}$, $0< r \leq 1$. Then
$$
\vert  \{ u \geq t \} \cap \sB_{r}   \vert \leq  C r^d (u(0)+ C_0 r^{\alpha})^{\varepsilon} t^{-\varepsilon} \;   \; \text{for every} \; t>0\,,
$$
for some  constants $C, \varepsilon$ as in \cref{T3.3}. In particular,
$$\norm{u}_{L^{\nicefrac{\varepsilon}{2}}(\sB_r)}\leq C (u(0)+ C_0 r^{\alpha}).$$
\end{theorem}

\begin{proof}
Choose $k\in\NN\cup\{0\}$ satisfying $2^{-k}< r\leq \frac{3}{2} 2^{-k}$.
Let $v(x)=u(\frac{1}{2^k}x)$. Then from the calculation of \cref{T3.2} it follows that
$$\cM^{-}_k v(x) \leq \frac{1}{2^{k\alpha}}\cM^{-} u(\frac{1}{2^k}x)\leq r^\alpha C_0\quad \text{in}\; \sB_2.$$
Multiplying $v$ with $\frac{\tilde\varepsilon_0}{v(0)+ r^\alpha C_0}$, it follows from \cref{T3.3} (modifying
the argument a bit)
\begin{equation}\label{ET3.4A}
|\{v\geq t\}\cap\sB_{\frac{3}{2}}|\leq C\, t^{-\varepsilon}\quad t>0\,.
\end{equation}
Hence, by our choice of $k$, we get
$$\vert  \{ u \geq t \} \cap \sB_{r}   \vert \leq  C r^d (u(0)+ C_0 r^{\alpha})^{\varepsilon} t^{-\varepsilon}.
$$ 
The second conclusion follows by integrating both sides of \eqref{ET3.4A} with respect to $t$.
\end{proof}

\section{H\"{o}lder regularity}\label{S-Holder}
Using the results developed in \cref{S-ABP}, in this section we establish an interior H\"{o}lder regularity.
The main theorem of this section is the following.
\begin{theorem}\label{T4.1}
Let $u$ be a bounded continuous function defined on $\Rd$ and satisfy
$$\cM^+ u \;\geq -C_0, \quad \cM^- u\leq\; C_0\quad \text{in}\; \sB_1,$$
for some constant $C_0$. Then $u\in\cC^\gamma(\sB_{\frac{1}{2}})$ and
$$\norm{u}_{\cC^\gamma(\sB_{\frac{1}{2}})}\;\leq\; C (\norm{u}_{L^\infty(\Rd)} + C_0),$$
where $\gamma, C$ depend only on $d, \lambda, \Lambda, \alpha, \varphi$.
\end{theorem}

We follow the approach of \cite{CS09} to prove \cref{T4.1}.
\cref{T4.1} would follow from the following result.
\begin{lemma}\label{L4.1}
Let $u$ be a continuous function satisfying
$$
-\frac{1}{2} \leq u \leq \frac{1}{2} \quad in \: \Rd, \quad 
\cM^{+} u \geq -\tilde\varepsilon, \quad \cM^{-}u \leq \tilde\varepsilon\quad \text{in}\; \sB_1\,.
$$
Then there is a $\gamma >0$ (depending on $\Lambda, \lambda, \alpha, \varphi$)
such that $u \in \cC^{\gamma}$ at the origin. In particular,
$$
\vert u(x) - u(0) \vert \leq  C \abs x^{\gamma}
$$
for some constant $C$.
\end{lemma}

\begin{proof}
Following \cite{CS09}
We show that there exists sequences $m_k$ and $M_k$ satisfying $m_k \leq  u \leq M_k$ in 
$\sB_{8^{-k}}$ and
$$
M_k - m_k = 8^{- \gamma k}.
$$
Then the results follows by choosing $C = 8^{\gamma}$.

For $k =0$ we choose $m_0 = -\frac{1}{ 2}$ and $ M_0 = \frac{1}{ 2}$. By assumption we have $m_0 \leq u \leq M_0$ in the whole space $\Rd$. We proceed to construct the sequences $M_k$ and $m_k$ by induction.
So by induction hypothesis we assume the construction of $m_j, M_j$ for $j=0,\ldots, k$.
We want to show that we can continue the sequences by finding $m_{k+1}$ and $M_{k+1}$.

Consider the ball  $\sB_{\frac{1}{8^{k+1}}}$. Then one of the following holds 
\begin{equation}\label{EL4.1A}
\vert \{ u \geq \frac{M_k + m_k}{2} \} \cap \sB_{\frac{1}{8^{k+1}}} \vert 
\;\geq\; \frac{1} {2}\vert \sB_{\frac{1}{8^{k+1}}}\vert\,;
\end{equation}
\begin{equation}\label{EL4.1B}
\vert \{ u \leq \frac{M_k + m_k}{2} \} \cap \sB_{\frac{1}{8^{k+1}}} \vert 
\;\geq\; \frac{1} {2}\vert \sB_{\frac{1}{8^{k+1}}}\vert\,.
\end{equation}
Suppose that \eqref{EL4.1A} holds. Define
$$
v(x) \df \frac{u(8^{-k} x ) - m_k}{\nicefrac{(M_k - m_k)}{2}}\,.
$$
Then that $v(x) \geq  0$ in $\sB_1$ and $\vert \{v \geq 1 \} \cap \sB_{\nicefrac{1}{8}} \vert \geq 		\nicefrac{ \vert B_{\nicefrac{1}{8}} \vert}{2}$. Moreover, since $\cM^{-}u \leq \tilde\varepsilon$ in $\sB_1$,
we get from the calculation in \cref{T3.2}
$$
\cM_{3k}^{-}v \leq \frac{8^{-k \alpha} \tilde\varepsilon}{\nicefrac{(M_k - m_k)}{2}} 
= 2 \tilde\varepsilon 8^{-k(\alpha - \gamma)} \leq 2 \tilde\varepsilon \quad \text{in} \; \sB_{8^k}\,,
$$
provided we set $\gamma\leq \alpha$. From the induction hypothesis, for any $j \geq 1$, we have
\begin{align*}
v \geq \frac{(m_{k-j} - m_k)}{\nicefrac{(M_k - m_k)}{2}} \geq \frac{(m_{k-j} -M_{k-j} + M_k - m_k)}{\nicefrac{(M_k - m_k)}{2}} 
&\geq -2 \cdot 8^{\gamma j} +2 \geq 2(1-8^{\gamma j})\quad  \text{ in} \; \sB_{8^j}\,.
\end{align*}
Thus $v(x) \geq \max\{-2(\vert 8x \vert^{\gamma} - 1),-2( 8^{(k+1)\gamma} - 1)\}\df-g(x)$ outside $\sB_1$. Letting $w(x) = \max(v,0)$
we also see that 
\begin{equation}\label{EL4.1C}
\cM_{3k}^{-} w \leq \cM_{3k}^{-} v + \cM_{3k}^{+} v^{-}.
\end{equation}

We claim that $\cM_{3k}^{+} v^{-} \leq 2 \tilde\varepsilon$ in $\sB_{\nicefrac{3}{4}}$, for all $k$, 
if we choose $\gamma$ small enough. For $x \in \sB_{\nicefrac{3}{4}}$, since $v^{-}(x) = 0$, we have 
$\delta(v^{-}, x, y) = \delta^{+}(v^{-}, x, y) = v^{-}(x+y)+ v^{-}(x-y)$ for all $y \in \Rd$, and 
by \eqref{E2.3},
$$
\cM_{3k}^{+} v^{-}(x) = \int_{\Rd} \frac{\Lambda \delta^{+}(v^{-}, x, y)}{\abs y^d}
 \left( \frac{2 - \alpha}{\abs y ^{\alpha}} + 8^{(\beta-\alpha)k}\varphi_0({1}/{ \abs y}) \right) \D{y}\,.
$$
If $\abs{y} < \frac{1}{4}$, then both $x+y$ and $x-y$ is in $\sB_1$ so $v^{-}(x+y) = v^{-}(x-y) = 0$.
This gives us
\begin{align*}
\cM_0^{+} v^{-}(x) &= \Lambda\int_{\{ \abs y \geq \frac{1}{4}\}} \frac{ v^{-}(x+y)+ v^{-}(x-y)}{\abs y^d} \left( \frac{ 2 - \alpha}{\abs y ^{\alpha}} + 8^{(\beta-\alpha)k}\varphi_0({1}/{ \abs y}) \right) \D{y}
 \\
 &\leq  \Lambda \int_{\{ \abs y \geq \frac{1}{4}\}} 
 \frac{  g^+(x+y) + g^+(x-y)}{\abs{y}^d} \left( \frac{ 2 - \alpha}{\abs y ^{\alpha}} + 8^{(\beta-\alpha)k}\kappa_\circ\varphi({1}/{\abs y}) \right) \D{y}
 \\
 &\leq 2 \Lambda \int_{\{ \abs y \geq \frac{1}{4}\}} 
 \frac{ g^+(x+y)}{\abs y^d} 
 \left( \frac{ 2 - \alpha}{\abs y ^{\alpha}} + 8^{(\beta-\alpha)k}\kappa_\circ\varphi(\frac{1}{ \abs y}) \right) \D{y}
  \\
 &\leq 4 \Lambda \int_{\{ \abs y \geq \frac{1}{4}\}} 
 \frac{ ( 32^{\gamma} \abs y^{\gamma} - 1)^+}{\abs y^d}  \frac{ 2 - \alpha}{\abs y ^{\alpha}}\, \D{y}
 + 4 \Lambda \int_{\{ \abs y \geq \frac{1}{4}\}} (8^{\gamma(k+1)}-1) 8^{(\beta-\alpha)k}\kappa_\circ
 \frac{1}{|y|^d}\varphi(\frac{1}{ \abs y})\D{y}\,
\\
&= I_1 + I_2.
\end{align*}
Notice that the function $f_{\gamma} = (32^{\gamma} \abs y^{\gamma} - 1)^+ \Ind_{\{ \abs y \geq \frac{1}{4}\}}$ decreases to $0$ as $\gamma \rightarrow 0$. Also, the function becomes integrable if we choose
$\gamma<\alpha$. Thus for a small $\gamma$ we have $I_1\leq \tilde\varepsilon$. So we calculate $I_2$.
We fix $\gamma<\alpha-\beta$. Define the function $h(t)=\log[(8^{\gamma t}-1) 8^{(\beta-\alpha)t}]$ for
$t>0$.  Note that
$$h^\prime(t) =\log 8 [-(\alpha-\beta) + \gamma \frac{8^{\gamma t}}{8^{\gamma t}-1}]<0,$$
for large $t$. So $h(t)$ attends its maximum and 
$$h^\prime(t)=0\; \Rightarrow 8^{\gamma t}-1=\gamma \frac{8^{\gamma t}}{\alpha-\beta}.$$
Thus, 
$$\max_{t\geq 1} e^{h(t)}\leq\; \gamma \sup_{t\geq 1} \frac{8^{(\gamma-\alpha+\beta) t}}{\alpha-\beta}
\to 0,$$
as $\gamma\to 0$. Thus using \eqref{EA2} we can choose $\gamma$ small enough to satisfy
$I_2<\tilde\varepsilon$, uniformly in $k$. This gives us the claim.

Using \eqref{EL4.1C} we obtain $M_{3k}^{-} w \leq 4 \tilde\varepsilon$ in $\sB_{\frac{3}{4}}$,
 provided $\gamma$ is small enough.
We also have 
$$ \vert \{w \geq 1\} \cap \sB_{\frac{1}{8}}\vert \geq \frac{|B_{\frac{1}{8}}|}{2}.$$
Given any point $x \in \sB_{\frac{1}{4}}$, we can apply \cref{T3.4} in
$B_{\frac{2}{4}}(x)$ to obtain
$$
C ( w(x) + 4\tilde\varepsilon)^{\varepsilon} \geq \vert \{ w > 1 \} \cap \sB_{\frac{1}{4}}(x)  \vert \geq \frac{1}{2} \vert \sB_{\frac{1}{8}}\vert\,.
$$
If we have chosen $\tilde\varepsilon$ small, this implies that $ w > \theta $ in $\sB_{\frac{1}{8}}$ for some
$ \theta > 0$. Thus if we let $M_{k+1} = M_k$ and $m_{k+1} = m_k + \theta \frac{( M_k - m_k)}{2}$, we have
$m_{k+1} \leq  u \leq  M_{k+1}$ in $B_{8^{-k-1}}$ . Moreover, 
$M_{k+1} - m_{k+1} =(1-\nicefrac{\theta}{2})8^{-\gamma k}$.
 So we must choose $\gamma$ and $\theta$ small and so that $(1-\nicefrac{\theta}{2}) = 8^{-\gamma}$ and we obtain $M_{k+1} - m_{k+1} =8^{-\gamma (k+1)} $.
\\
On the other hand, if \eqref{EL4.1B} holds, we define
$$
v(x) = \frac{M_k - u(8^{-k} x ) }{\nicefrac{(M_k - m_k)}{2}}
$$
\\
and continue in the same way using that $\cM^{+}u \geq - \tilde\varepsilon$.

\end{proof}

\section{Harnack inequality}\label{S-Har}
In this section and the next section we discuss Harnack's inequality and boundary Harnack inequality.
This will be done for a smaller class of operators. Let $\tilde\sL\subset\sL$ be the set of operators containing
kernel function $k$ satisfying
\begin{equation}\label{EA4}
\lambda \left( \frac{2- \alpha}{{\abs y}^{\alpha}} + 
\varphi({1}/{\abs y}) \right)  
\;\leq\; k(y) \leq  \Lambda \left( \frac{2- \alpha}{{\abs y}^{\alpha}} + 
\varphi({1}/{\abs y}) \right) \quad \text{for some}\; \; \alpha\in (\beta, 2),\tag{A4}
\end{equation}
and $\varphi$ is non-decreasing.
The associated extremal operators are denoted by $\tilde\cM^{\pm}$. In particular,
\begin{align*}
\tilde\cM^{+}u(x) &= \int_{\Rd} \frac{\Lambda \delta^{+}(u,x,y)}{\abs{y}^{d}}
 \left( \frac{ 2 - \alpha}{\abs{y}^{\alpha}} + \varphi({1}/{ \abs y}) \right ) - \frac{\lambda \delta^{-}(u,x,y)}{\abs{y}^{d}} \left( \frac{ 2 - \alpha}{\abs{y}^{\alpha}}  + \varphi({1}/{ \abs y}) \right) \D{y}\,,
\\
\tilde\cM^{-}u(x) &= \int_{\Rd} \frac{\lambda \delta^{+}(u,x,y)}{\abs{y}^{d}} \left( \frac{ 2 - \alpha}{\abs y ^{\alpha}} + \varphi({1}/{ \abs y})  \right ) - \frac{\Lambda \delta^{-}(u,x,y)}{\abs{y}^{d}} \left( \frac{ 2 - \alpha}{\abs{y}^{\alpha}} + \varphi({1}/{ \abs y}) \right ) \D{y}\,.
\end{align*}
It is also evident that $\cM^{-} u \leq \tilde \cM^{-} u\leq \tilde\cM^+ u\leq \cM^+$.
It should be observed that we do not require weak lower scaling property
 on $\varphi$ (compare with \cite{KL20}).
 For instance, 
$\varphi(r)=\log(1+ r^\beta)$ does not satisfy weak lower scaling i.e. 
there is no $\mu>0$ so that $\varphi(s r)\gtrsim s^\mu \varphi(r)$
for $s\geq 1, r>0$. But it does satisfy a weak upper scaling property since for every $s\geq 1$,
$$1+s^\beta r^\beta \leq (1+ r^\beta)^{s^\beta}\;\Rightarrow\; \varphi(s r)\leq s^\beta \varphi(r)\,.$$

Our main result of this section is the following
\begin{theorem}\label{T5.1}
Let $u$ be a  non-negative function satisfying
$$\tilde\cM^{+}u \geq -C_0,\quad \text{and}\quad \tilde\cM^{-}u \leq C_0\quad \text{in}\; \sB_2.$$
Then $u(x) \leq  C ( u(0) + C_0)$ for every $x \in \sB_{\frac{1}{2}}$, for some constant
$C$ dependent only on $\lambda, \Lambda, \alpha, \varphi$.
\end{theorem}

\begin{proof}
We again follow the idea of \cite{CS09}. 
Dividing by $u(0) + C_0$, it is enough to consider $u(0) \leq  1$ and $C_0 = 1$. Fix
$\varepsilon > 0$ from \cref{T3.4} and
let $\gamma = \frac{d}{\varepsilon}$. 
Let
$$
t\df\min\{s\; :\; u(x) \leq h_{s}(x) \df s(1 - \abs x)^{- \gamma} \;\; \text{for all} \; x \in \sB_1\}.
$$
Let $x_0 \in \sB_1$ be such that $u(x_0) = h_t (x_0)$. Let  
$\eta = 1 - \vert x_0 \vert$ be the distance of $x_0$ from $\partial \sB_1$.
We show that $t<C$ for some universal $C$ which in trun, implies that $u(x) < C (1 -  \abs{x})^{-\gamma}$.
This would prove our result.

For $r = \frac{\eta}{2}$, we  estimate the portion of the ball $\sB_r(x_0)$ covered by 
$\{u < \frac{u(x_0)}{2}\}$ and $\{ u > \frac{u(x_0)}{2}\} $. 
Define $A \df  \{ u > \frac{u(x_0)}{2} \} $. Using \cref{T3.4} we then obatin
$$
\vert A \cap \sB_1 \vert \leq C \bigg\lvert \frac{2}{u(x_0)} \bigg\rvert^{\varepsilon} \leq C t^{-\varepsilon}
\eta^d\,,
$$
whereas $\vert \sB_r \vert = \omega_d (\eta/2)^d$. In particular,
\begin{equation}\label{ET5.1A0}
\bigg\lvert \bigg\lbrace   u > \frac{u(x_0)}{2}   \bigg\rbrace  \cap \sB_r(x_0)  \bigg\rvert \leq 
C t^{- \varepsilon} \vert \sB_r \vert .
\end{equation}
So if $t$ is large,
$A$ can cover only a small portion of $\sB_r(x_0)$. We shall show that for some $\delta>0$,
independent of $t$ we have 
$$ \vert \{ u \leq \frac{u(x_0)}{2} \} \cap \sB_r(x_0) \vert \leq  (1- \delta) \vert \sB_r \vert$$
which will provide an upper bound on $t$ completing the proof.

We start by estimating $ \vert \{ u \leq \frac{u(x_0)}{2} \} \cap \sB_{\theta r} (x_0) \vert $ for $ \theta > 0$ small. For every $ x \in \sB_{\theta r}(x_0)$ we have
$$
u(x) \leq\, h_t (x) \,\leq\, t \left( \frac{2\eta - \theta \eta}{2} \right)^{- \gamma} 
\,\leq\, u(x_0) \left( 1 - \frac{\theta }{2} \right)^{- \gamma}\,, 
$$
with $\left( 1 - \frac{\theta }{2} \right)$ close to 1.
Define
$$
v(x) \df \left( 1 - \frac{\theta }{2} \right)^{- \gamma} u(x_0) - u(x)\,.
$$
Then that $v \geq 0$ in $\sB_{ \theta r}(x_0)$, and also $\tilde\cM^{-} v \leq 1$ as 
$\tilde\cM^{+} u \geq 1$.
We would like to
apply \cref{T3.4} to $v$ but $v$ need not be non-negative in the whole
space $\Rd$. Thus we consider $w=v^+$ and find an upper bound of $\tilde\cM^{-} w$.

We already know that
$$
\tilde\cM^{-}v(x) = \int_{\Rd} \frac{\lambda \delta^{+}(v,x,y)}{\abs y ^{d}} 
\left( \frac{ 2 - \alpha}{\abs{y}^{\alpha}} + \varphi({1}/{ \abs y})  \right )
 - \frac{\Lambda \delta^{-}(v,x,y)}{\abs y ^{d}} \left( \frac{ 2 - \alpha}{\abs y ^{\alpha}} 
 + \varphi({1}/{ \abs y}) \right ) \D{y} \leq\, 1\,.
$$
Therefore, for $x \in \sB_{\frac{\theta r}{4}}(x_0)$
\begin{align}\label{ET5.1A}
\tilde\cM^{-}w(x) &= \int_{\Rd} \frac{\lambda \delta^{+}(w,x,y)}{\abs y ^{d}} \left(\frac{ 2 - \alpha}{\abs y ^{\alpha}} + \varphi({1}/{ \abs y})  \right) - \frac{\Lambda \delta^{-}(w,x,y)}{\abs y ^{d}} \left( \frac{ 2 - \alpha}{\abs y ^{\alpha}} + \varphi({1}/{ \abs y}) \right ) \D{y}\nonumber
\\
&\leq\, 1 +2 \int_{\Rd \cap \{ v(x+y) <0 \}}  - \Lambda \frac{v(x+y)}{\abs y^{d}}
 \left( \frac{ 2 - \alpha}{\abs y ^{\alpha}} + \varphi({1}/{ \abs y}) \right ) \D{y}\nonumber
 \\
&\leq 1 +2 \int_{\Rd \setminus \sB_{\frac{\theta r}{2}}(x_0 - x)}  
\Lambda \frac{(u(x+y) - (1 - \frac{\theta}{2})^{- \gamma} u(x_0))^{+}}{\abs y^{d}}
 \left( \frac{ 2 - \alpha}{\abs y ^{\alpha}} + \varphi({1}/{ \abs y}) \right ) \D{y}\,.
\end{align}
So to  find an upper bound we must compute the second expression.
Let us consider the largest value $ \tau > 0 $ such that $ u(x) \geq g_{\tau} \df \tau \left(1 - \vert 4x \vert^{2}\right)$. There must be a point $x_1 \in  \sB_{\frac{1}{4}}$ such that $u(x_1) =\tau (1 - \vert 4x \vert^{2})$. The value of $\tau$ cannot be larger than $1$ since $u(0) \leq 1$. Also truncate $g_{\tau}$ and define $\hat{g}_{\tau} \df g_{\tau} \Ind_{\sB_{\frac{1}{3}}}$ which implies $u(x) \geq \hat{g}_{\tau}(x) \geq g_{\tau}(x)$ for all $x \in \Rd$ and $u(x_1)= \hat{g}_{\tau}(x_1) = g_{\tau}(x_1)$. Thus we have the upper bound
\begin{align*}
&\int_{\Rd} \frac{ \delta^{-}(u,x_1,y)}{\abs{y}^{d}} \left( \frac{ 2 - \alpha}{\abs y ^{\alpha}} + \varphi(1/{ \abs y}) \right ) \D{y}
 \\
&\leq \int_{\Rd} \frac{ \delta^{-}(\hat{g}_{\tau},x_1,y)}{\abs y ^{d}} \left(\frac{ 2 - \alpha}{\abs y ^{\alpha}} + \varphi(1/{ \abs y}) \right ) \D{y}
 \\
&\leq \int_{\sB_1} \frac{ \delta^{-}(g_{\tau},x_1,y)}{\abs y ^{d}} \left( \frac{ 2 - \alpha}{\abs y ^{\alpha}} + \varphi(1/{ \abs y}) \right ) \D{y}    +\int_{\Rd \setminus \sB_1} \frac{ 32}{\abs y ^{d}} 
\left( \frac{ 2 - \alpha}{\abs y ^{\alpha}} + \varphi(\frac{1}{ \abs y}) \right ) \D{y}
 \\
&\leq \int_{\sB_1} \frac{ 32 \abs y^{2}}{\abs y ^{d}} \left(\frac{ 2 - \alpha}{\abs y ^{\alpha}} + 
\frac{\kappa_\circ}{ \abs y^{\beta}} \varphi(1) \right ) \D{y}    + C_1 \leq C_2\,,
\end{align*}
for some constants $C_1, C_2$ dependent only on $(d, \alpha ,  \varphi)$,
where we used following inequality
$$
\hat{g}_{\tau}(x_1 + y) + \hat{g}_{\tau}(x_1 - y) - 2\hat{g}_{\tau}(x_1) \geq \tau \left( 2 \vert 4x_1 \vert^{2} - \vert 4(x_1+y) \vert^{2} -\vert 4(x_1-y) \vert^{2}  \right)=32|y|^2,
$$
for $y \in \sB_1$. Since $\tilde\cM^{-}u(x_1) \leq 1$, we get using the above  estimate that
$$
\int_{\Rd} \frac{ \delta^{+}(u,x_1,y)}{\abs y ^{d}} \left( \frac{2 - \alpha }{\abs y ^{\alpha}} + 
\varphi({1}/{ \abs y}) \right ) \D{y}  \leq\; C\,.
$$
In particular, since $u(x_1) \leq 1$ and $u(x_1 -y) \geq 0$,
$$
\int_{\Rd} \frac{ (u(x_1 + y)- 2)^{+}}{\abs y ^{d}} \left( \frac{ 2 - \alpha}{\abs y ^{\alpha}}
 + \varphi({1}/{ \abs y}) \right ) \D{y} \,\leq\, C\,.
$$
We use this estimate to compute the RHS of \eqref{ET5.1A}. 
Without any loss of generality we may assume that $u(x_0) > 2$, since otherwise t would not be large. We can use the inequality above to get following estimate
 \begin{align*}
& 2 (2 - \alpha)\int_{\Rd \setminus \sB_{\frac{\theta r}{2}}(x_0 - x)} \Lambda \frac{(u(x+y) - (1 - \frac{\theta}{2})^{- \gamma} u(x_0))^{+}}{\abs y^{d + \alpha}}  \D{y} 
\\
&\quad \leq 2 (2 - \alpha)  \int_{\Rd \setminus \sB_{\frac{\theta r}{2}}(x_0 - x)}  
\Lambda \frac{(u(x_1 + x+y - x_1) - 2)^{+}}{\vert y + x - x_1 \vert^{d+\alpha}} \quad \frac{\vert y + x - x_1 \vert^{d+\alpha}}{\abs y^{d + \alpha}}  \D{y}
\\
&\quad\leq C (\theta r)^{-d-\alpha},
 \end{align*}
here we used the fact that $ y\notin \sB_{\frac{\theta r}{2}} (x_0 - x)$ implies $y \notin \sB_{\frac{\theta r}{4}}$. Again, a simple calculation gives
$$
\frac{\vert y + x - x_1 \vert}{\abs y} \leq \frac{\vert y \vert + \vert x - x_1 \vert}{\abs y} \leq 12 (\theta r)^{-1}
$$ 
and using the monotonicity property of $\varphi$,
$$
\frac{\varphi \left({1}/{ \abs y} \right )}{\varphi \left({1}/{ \vert y + x - x_1 \vert} \right )}
 \leq \frac{\varphi \left(  \frac{\vert y + x - x_1 \vert }{\vert y + x - x_1 \vert}  \frac{1}{ \abs y} \right )}{\varphi \left(\frac{1}{ \vert y + x - x_1 \vert} \right )} \leq  \kappa_\circ \left( 12 (\theta r)^{-1} \right)^{\beta},
$$
by \eqref{EA1}. This gives us
\begin{align*}
& 2\Lambda \int_{\Rd \setminus \sB_{\frac{\theta r}{2}}(x_0 - x)} \frac{(u(x+y) - (1 - \frac{\theta}{2})^{- \gamma} u(x_0))^{+}}{\abs y^{d}}   \varphi \left(\frac{1}{ \abs y} \right ) \D{y} 
\\
&\quad \leq 2 \Lambda  \int_{\Rd \setminus \sB_{\frac{\theta r}{2}}(x_0 - x)}  \frac{(u(x_1 + x+y - x_1) - 2)^{+}}{\vert y + x - x_1 \vert^{d}} \quad  \varphi \left({1}/{ \vert y + x - x_1 \vert} \right )  
\\
 &\hspace{10em}  \frac{\vert y + x - x_1 \vert^{n}}{\abs y^{n}} 
 \quad  \frac{\varphi \left({1}/{ \abs y} \right )}{\varphi \left({1}/{ \vert y + x - x_1 \vert} \right )} \D{y}
 \\
&\quad \leq C \kappa_\circ (\theta r)^{-d-\beta} \leq C (\theta r)^{-d-\alpha}\,.
\end{align*}
Thus we obtain from \eqref{ET5.1A}
$$ 
\tilde\cM^{-}w \,\leq\, C_1(\theta r)^{-d-\alpha} \quad \text{in}  \; \sB_{\frac{\theta r}{4}}(x_0).
$$
We apply \cref{T3.4} to $w$ in $B_{\nicefrac{\theta r}{4}}(x_0)$. Recalling that $w(x_0) = ((1- \nicefrac{\theta}{2})^{ - \gamma}-1) u(x_0) $, we have
\begin{align*}
  &\bigg\lvert \bigg\lbrace   u \leq \frac{u(x_0)}{2}   \bigg\rbrace  \cap \sB_{\frac{\theta r}{8}}(x_0)  \bigg\rvert	
  \\		
 & \qquad =   \bigg\lvert \bigg\lbrace   w \geq  u(x_0) ((1- \nicefrac{\theta}{2})^{ - \gamma}-\nicefrac{1}{2})   \bigg\rbrace  \cap \sB_{\frac{\theta r}{8}}(x_0)  \bigg\rvert	
 \\
 & \qquad \leq C (\theta r)^d \left( \left((1- \nicefrac{\theta}{2})^{ - \gamma}-1 \right) u(x_0)  +
  C_1 (\theta r)^{-d - \alpha} (\theta r)^{\alpha} \right)^{\varepsilon}
   \cdot \left[ u(x_0) ((1- \nicefrac{\theta}{2})^{ - \gamma}-\nicefrac{1}{2})  \right]^{- \varepsilon}
  \\
  & \quad \quad \leq C (\theta r)^d \left( \left((1- \nicefrac{\theta}{2})^{ - \gamma}-1 \right)^{\varepsilon} + \theta^{-d \varepsilon} t^{- \varepsilon} \right)\,.
\end{align*}
Now let us choose $\theta > 0$ small enough (independent of $t$) to satisfy
$$
C (\theta r)^d  \left((1- \nicefrac{\theta}{2})^{ - \gamma}-1 \right)^{\varepsilon} \leq\;
 \frac{1}{4} \vert \sB_{\frac{\theta r}{8}}(x_0) \vert\,.
$$
With this choice of $\theta$ if $t$ becomes large,  then we  also have
$$
C (\theta r)^d \theta^{-d \varepsilon} t^{- \varepsilon} \leq \frac{1}{4} \vert \sB_{\frac{\theta r}{8}}(x_0) \vert\,,
$$
and hence,
$$
\bigg\lvert \bigg\lbrace   u \leq \frac{u(x_0)}{2}   \bigg\rbrace  \cap \sB_{\frac{\theta r}{8}}(x_0)  \bigg\rvert \leq \frac{1}{2} \vert \sB_{\frac{\theta r}{8}}(x_0) \vert\,. 
$$
This of course, implies that 
$$
\bigg\lvert \bigg\lbrace   u > \frac{u(x_0)}{2}   \bigg\rbrace  \cap \sB_{\frac{\theta r}{8}}(x_0)  \bigg\rvert \geq C_2 \vert \sB_r \vert,
$$
but this is contradicting to \eqref{ET5.1A0}. Therefore $t$ cannot be large and we finish the proof.
\end{proof}

Mimicking \cref{T5.1} we also obtain the following result which will be useful to establish a boundary Harnack
property. The following also known as the \textit{half Harnack} inequality for subsolutions.
\begin{theorem}\label{T5.2}
Let $u$ be a function continuous in $\overline{\sB_1}$, and satisfy
$$
\int_{\Rd} \frac{\vert u(y)\vert}{1 + \abs y^{d + \alpha}}\, \D{y} 
+ \int_{\Rd} \frac{\vert u(y)\vert}{1 + \abs y^d (\varphi ({1}/{\abs y}))^{-1}}\, \D{y} \leq\, C_0,
$$
and
$$
\tilde\cM^{+} u \geq - C_0 \quad \text{in}\;  \sB_1\,.
$$
Then
$$
u(x) \leq C\,C_0 \,,
$$
for every $x \in \sB_{\frac{1}{2}}$, where the constant $C >0$ depends only on $(d, \lambda,  \Lambda,  \alpha, \varphi)$.
\end{theorem}

\begin{proof}
We follow the approach of \cite{CS11} and \cref{T5.1}. Dividing $u$ by $ C_0$, it is enough to consider $C_0 = 1$. Also, without any loss of generality we may assume that $u$ is positive somewhere in $\sB_1$. Otherwise,
there is nothing to prove.
As before we consider the minimum value of $t$ such that
$$
u(x) \leq h_{t}(x) \df t(1 - \abs x)^{- d} \quad for \; every \; x \in \sB_1,
$$
and find $x_0\in\sB_1$ with $u(x_0)=h_t(x_0)$. Denote by $\eta=1-\abs{x_0}$,  $r=\eta/2$ and
$A=\{u> u(x_0)/2\}$. As shown in \cref{T5.1}, we need to find an upper bound of $t$.

By assumption, we have $u \in L^1(\sB_1)$ and thus
$$
\vert A \cap \sB_1 \vert \leq C \bigg\lvert \frac{2}{u(x_0)} \bigg\rvert \leq C t^{-1} \eta^d\,,
$$
whereas $\vert \sB_r \vert = C \eta^d$, so if $t$ is large, $A$ can cover only a 
small portion of $\sB_r(x_0)$ at most. In particular,
\begin{equation}\label{ET5.2A}
\bigg\lvert \bigg\lbrace   u > \frac{u(x_0)}{2}   \bigg\rbrace  \cap B_r(x_0)  \bigg\rvert\leq C t^{-1} \vert B_r \vert .
\end{equation}
We define
$$
v(x) = \left( 1 - \frac{\theta }{2} \right)^{-d} u(x_0) - u(x)
$$
for small $\theta>0$, and observe that $v \geq 0$ in $\sB_{ \theta r}(x_0)$. Let $w=v^+$.
Repeating the arguments of \cref{T5.1} we find, for $x \in \sB_{\frac{\theta r}{4}}(x_0)$, that
\begin{align*}
\tilde\cM^{-} w(x) 
&\leq 1 +2 \int_{\Rd \cap \{ v(x+y) <0 \}}  - \Lambda \frac{v(x+y)}{\abs y^{d}} 
\left( \frac{ 2 - \alpha}{\abs y ^{\alpha}} + \varphi({1}/{ \abs y}) \right ) \D{y}\,,
\end{align*}
since $v \geq 0$ in $\sB_{ \theta r}(x_0)$ and $x \in \sB_{\frac{\theta r}{4}}(x_0)$, we will have $x+y$ and $x-y$ both in $\sB_{ \theta r}(x_0)$ for all $ y \in \sB_{\frac{\theta r}{2}}$. Now we need to estimate the
integral on the RHS of the above. Note that $u$ need to be non-negative here and thus we can not apply
the technique of cut-off function as done in \cref{T5.1}. So we use the integral condition imposed on
$u$.
\begin{align}\label{ET5.2B}
\tilde\cM^{-} w(x) 
&\leq 1 +2 \int_{\Rd \setminus \sB_{\frac{\theta r}{2}}}  \Lambda \frac{(u(x+y) - (1 - \frac{\theta}{2})^{- d} u(x_0))^{+}}{\abs y^{d}} \left( \frac{ 2 - \alpha}{\abs y ^{\alpha}} + 
\varphi({1}/{ \abs y}) \right ) \D{y}\nonumber
\\
&\leq 1 + 2 \Lambda \int_{\Rd \setminus \sB_{\frac{\theta r}{2}}}   
\frac{u^{+}(x+y)}{\abs y^{d}} \left( \frac{ 2 - \alpha}{\abs y ^{\alpha}} 
+ \varphi({1}/{ \abs y}) \right ) \D{y}\nonumber
\\
&\leq 1 +2 \Lambda \int_{\Rd \setminus \sB_{\frac{\theta r}{2}}(x)}   
\frac{\vert u(x) \vert}{ \vert x - y \vert^{d}} \left( \frac{ 2 - \alpha}{\vert x - y \vert^{\alpha}} + \varphi({1}/{ \vert x -y \vert}) \right ) \D{y}\,.
\end{align}
Using $\vert x-y \vert \geq \frac{\theta r}{2}$ and $\abs x < 1$ we obtain the following estimates
\begin{align*}
\frac{1}{\vert x-y \vert^{d +\alpha}} &= \frac{1}{1 + \abs y^{d+\alpha}} 
\; \cdot \;  \frac{1 + \abs y^{d+\alpha}}{\vert x-y \vert^{d+\alpha}}
  \\
 & \leq \frac{1}{1 + \abs y^{d+\alpha}} \; \cdot \; \left[ \frac{1}{\vert x-y \vert^{d+\alpha}} + \left( \frac{ \abs x + \vert x- y \vert}{\vert x-y \vert} \right)^{d+\alpha} \right] 
 \\
 & \leq \frac{1}{1 + \abs y^{d+\alpha}} \left( \frac{\theta r}{2} \right)^{-d-\alpha} \left[ 1 + 2^{d+\alpha} \right] \leq C (\theta r)^{-d -\alpha}  \frac{1}{1 + \abs y^{d+\alpha}}\,,
\end{align*}
and, since $\frac{\abs y}{\vert x-y \vert} \leq 1+ \frac{\abs x}{\vert x-y \vert}$,
\begin{align*}
\varphi({1}/{\vert x-y \vert}) \leq \kappa_\circ \left( 1+ \frac{\abs x}{\vert x-y \vert} \right)^{\beta} \varphi({1}/{\abs y}), \quad \varphi(1/|x-y|)\leq\varphi(2/r\theta)\leq\,\kappa_\circ (2/r\theta)^\beta \varphi(1),
\end{align*}
giving us
\begin{align*}
\frac{1}{\vert x-y \vert^d (\varphi({1}/{\vert x-y \vert}))^{-1}} 
&= \frac{1}{1+\abs y^d (\varphi({1}/{\abs y}))^{-1}} 
\left[ \frac{\varphi({1}/{\vert x-y \vert})}{\vert x-y \vert^d }
  + \frac{\abs y^d}{\vert x-y \vert^d} \frac{\varphi(1/{\vert x-y \vert})}{\varphi({1}/{\abs y})}\right]
\\
&\leq \frac{1}{1+\abs y^d (\varphi({1}/{\abs y}))^{-1}} \left[ \frac{\kappa_\circ (2/r\theta)^\beta\varphi(1)}{\vert x-y \vert^d}
 +  \kappa_\circ \left( 1+ \frac{\abs x}{\vert x-y \vert} \right)^{ d + \beta} \right]
\\
& \leq C_1 \frac{1}{1+\abs y^d (\varphi({1}/{\abs y}))^{-1}} (\theta r)^{-d-\beta}
\\
& \leq C_1 \frac{1}{1+\abs y^d (\varphi({1}/{\abs y}))^{-1}} (\theta r)^{-d-\alpha},
\end{align*}
for some constant $C_1$ dependent on $d,  \varphi$. Using these estimates in \eqref{ET5.2B} we thus
obtain
$$\tilde\cM^{-} w(x)\,\leq C_3\, (\theta r)^{-d-\alpha}\quad \text{in}\; \sB_{\frac{r\theta}{4}}(x_0)\,.$$
Now we repeat the arguments of \cref{T5.1} and get a contradiction to \eqref{ET5.2A} if $t$ is large.
This completes the proof.
\end{proof}

\section{Boundary Harnack estimate}\label{S-bHar}
We prove a boundary Harnack property in this section for operators in $\tilde\sL$ introduced in \cref{S-Har}.
Being inspired from \cite{RS19} we prove the following result
\begin{theorem}\label{T6.1}
Let $ \Omega\subset \Rd$ be any open set. Assume that there is
$x_0 \in \sB_{\frac{1}{2}}$ and $\varrho > 0$ such that 
$\sB_{2 \varrho}(x_0) \subset \Omega \cap  \sB_{\frac{1}{2}}$.
Then there exists $\delta > 0$, dependent only on $(d, \alpha , \varrho, \varphi, \lambda, \Lambda)$, such that the following statement holds.

Let $u_1, u_2 \in C(\sB_1)$ be viscosity solutions of
\begin{equation}\label{ET6.1A}
\begin{split}
 \tilde\cM^{+} (a u_1 + b u_2) &\geq - \delta (\abs a + \abs b)    \quad \text{in } \quad    \sB_1 \cap \Omega,
 \\
    u_1 = u_2 &= 0  \quad   \text{ in} \quad \sB_1 \setminus \Omega\,,
  \end{split}
  \end{equation}
for all $a, b \in \RR$, and such that
\begin{equation}\label{ET6.1B}
u_i \geq 0 \quad \text{in} \quad \Rd , \qquad \int_{\Rd} \frac{u_i(y)}{1 + \abs y^{d+\alpha}} dy + 
\int_{\Rd} \frac{\vert u_i(y)\vert}{1 + \abs y^d (\varphi({1}/{\abs y}))^{-1}}\, \D{y}\, = 1\,.
\end{equation}
Then
$$
C^{-1} u_2 \leq u_1 \leq C u_2 \quad \quad in \quad \sB_{\frac{1}{2}},
$$
where the constant $C$ depends only on $ (d, \alpha , \varrho, \varphi, \lambda, \Lambda)$.
\end{theorem}

\cref{T6.1} is bit stronger than the boundary Harnack principle. 
To see it suppose that for some $L\in\tilde\sL$ 
we have $L u_i=0$ in $\sB_1 \cap \Omega$, in viscosity sense, and $u_i=0$ in $\sB_1 \setminus \Omega$.
Then clearly \eqref{ET6.1A} holds for all $a, b\in\RR$ (cf. \cite[Theorem~5.9]{CS09}). Furthermore, if \eqref{ET6.1B} holds, then
\cref{T6.1} gives us
$$
C^{-1} u_2 \leq u_1 \leq C u_2 \quad \quad in \quad \sB_{\frac{1}{2}}\,.
$$

To prove \cref{T6.1} we need \cref{L6.1,L6.2} below.
\begin{lemma}\label{L6.1}
Assume that $u \in \cC(\sB_1)$ and satisfies $\tilde\cM^{-} u \leq C_0$ in $\sB_1$ in viscosity sense. 
In addition, assume that $u \geq 0$ in $\Rd$. Then 
$$
\int_{\Rd} \frac{u(y)}{1 + \abs y^{d + \alpha}} \D{y} 
+ \int_{\Rd} \frac{u(y)}{1 + \abs y^d (\varphi ({1}/{\abs y}))^{-1}} \D{y} \,\leq\, C\left( \inf_{\sB_{\frac{1}{2}}} u + C_0 \right),
$$
where the constant $C$ depends only on $(d,\lambda, \Lambda, \alpha , \varphi)$.
\end{lemma}

\begin{proof}
We need few basic estimates. We show that there exists a constant $\kappa>0$ such that
for any $x_0\in\sB_{\frac{3}{4}}$ and $z\in\Rd$ we have
\begin{align}
\vert x_0 - z \vert^{d + \alpha} &\leq \kappa ( 1 + \abs z^{d+\alpha}),\label{EL6.1A}
\\
\vert x_0 - z \vert^{d } (\varphi({1}/{\vert x_0 - z \vert}))^{-1} &\leq \kappa 
( 1 + \abs z^{d} (\varphi({1}/{\abs z}))^{-1})\,.\label{EL6.1B}
\end{align}
\eqref{EL6.1A} is trivial since $\vert x_0 + z \vert \leq 1 + \abs z$ implies 
$\vert x_0 - z \vert^{d + \alpha} \leq 2^{d + \alpha} ( 1 + \abs z^{d+\alpha})$.
On the other hand 
$$
\frac{1}{\abs z} \leq \left( \frac{1 + \abs z}{\abs z} \right) \frac{1}{\vert x_0 - z \vert}\,,
$$
implies
\begin{align*}
\varphi \left( \frac{1}{\abs z} \right) 
&\leq \varphi \left( \left( \frac{1 + \abs z}{\abs z} \right) \frac{1}{\vert x_0 - z \vert} \right)
\\
&\leq \kappa_\circ  \left( \frac{1 + \abs z}{\abs z} \right)^{\beta} \varphi \left( \frac{1}{\vert x_0 - z \vert} \right)\,.
\end{align*}
Thus
\begin{equation}\label{EL6.1C}
(\varphi\left({1}/{\vert x_0 - z \vert} \right))^{-1} 
\leq \kappa_\circ  \left( \frac{1 + \abs z}{\abs z} \right)^{\beta} (\varphi\left({1}/{\abs z} \right))^{-1}\,.
\end{equation}
Let $\abs z \leq 1$. Then using \eqref{EL6.1C} we get
\begin{align*}
\vert x_0 - z \vert^d  \left(\varphi({1}/{\vert x_0 - z \vert})\right)^{-1}
&\leq (1+ \abs z)^{d+\beta} \left(\varphi({1}/{\vert x_0 - z \vert})\right)^{-1}
\\
&\leq 2^{d+\beta} (1+ \abs z^{d+\beta} ) \left(\varphi({1}/{\vert x_0 - z \vert})\right)^{-1}
\\
&\leq 2^{d+\beta} \left( \left(\varphi({1}/{\vert x_0 - z \vert})\right)^{-1} + \abs z^{d+\beta}  \left(\varphi({1}/{\vert x_0 - z \vert})\right)^{-1} \right)
\\
&\leq 2^{d+\beta} \left( \kappa + \abs z^{d+\beta}  \left(\varphi({1}/{\vert x_0 - z \vert})\right)^{-1} \right)
\\
&\leq 2^{d+\beta} \left( \kappa + \kappa_\circ 2^{\beta} \abs z^{d}  (\varphi( 1/{\abs z}))^{-1} \right)\,,
\end{align*}
where $\kappa= \max \left\lbrace (\varphi({1}/{2}))^{-1}, 1 \right\rbrace$. Here we use 
$(\varphi\left({1}/{\vert x_0 - z \vert} \right))^{-1} \leq (\varphi({1}/{2}))^{-1}$, since $\vert x_0 - z \vert < 2$.
Again, for $\abs z > 1$, we use \eqref{EL6.1C} to obtain
\begin{align*}
\vert x_0 - z \vert^d  \left(\varphi({1}/{\vert x_0 - z \vert})\right)^{-1}
&\leq ( 1+ \abs z )^d \kappa_\circ  \left( \frac{1 + \abs z}{\abs z} \right)^{\beta} 
(\varphi\left({1}/{\abs z} \right))^{-1}
 \\
&\leq \kappa_\circ 2^{\beta} (1 + \abs z^d)(\varphi\left({1}/{\abs z} \right))^{-1}
\\
&\leq \kappa_\circ 2^{\beta+1} \abs z^d	(\varphi\left({1}/{\abs z} \right))^{-1}
\\
&\leq \kappa_\circ 2^{\beta+1} (1+ \abs z^d	(\varphi\left({1}/{\abs z} \right))^{-1})\,.
\end{align*}
This gives us \eqref{EL6.1B}.

Let $\chi \in \cC_c^{\infty}(\sB_{\frac{3}{4}})$ be such that $0\leq \chi\leq 1$
and $\chi =1$ in $\sB_{\frac{1}{2}}$.
Let $t >0$ be the maximum value for which $u \geq t\chi$. It is easily seen
that $t \leq \inf_{\sB_{\frac{1}{2}}} u$. 
Since $u$ and $\chi$ are continuous in $\sB_1$ there exists $x_0 \in \sB_{\frac{3}{4}}$ such that 
$u(x_0) = t\chi(x_0)$.
We also get
$$
\tilde\cM^{-}(u - t\chi)(x_0) \leq \tilde\cM^{-}u(x_0) - t \tilde\cM^{-}\chi \leq C_0 + C t
\quad \text{in}\; \sB_1\,.
$$
On the other hand, since $u - t \chi \geq 0$ in $\Rd$ and $(u - t\chi)(x_0) = 0$, we also obtain
from \eqref{EL6.1A}-\eqref{EL6.1B}
\begin{align*}
\tilde\cM^{-}(u - t\chi)(x_0) 
&= 2 \lambda \int_{\Rd} \frac{u(z) - t \chi(z)}{\vert x_0 - z \vert^d} \left( \frac{ 2 - \alpha}{\vert x_0 - z \vert^{\alpha}} + \varphi({1}/{ \vert x_0 - z \vert}) \right ) \D{z}
\\
&\geq 2  (2 - \alpha) \lambda \int_{\Rd} \frac{u(z) - t \chi(z)}{\vert x_0 - z \vert^{d + \alpha}}  \D{z}
 + 2 \lambda \int_{\Rd} \frac{u(z) - t \chi(z)}{\vert x_0 - z \vert^d (\varphi({1}/{ \vert x_0 - z \vert}))^{-1}  }  \D{z}
  \\
&\geq 2  (2 - \alpha) \lambda \kappa \int_{\Rd} \frac{u(z) - t \chi(z)}{  1 + \abs z^{d + \alpha} } \D{z}
 + 2 \lambda \kappa \int_{\Rd} \frac{u(z) - t \chi(z)}{ 1 + \abs z^{d} (\varphi({1}/{\abs z}))^{-1} }  \D{z}
  \\
&\geq C_1 \left( \int_{\Rd} \frac{u(z) }{  1 + \abs z^{d + \alpha} } \D{z} + 
 \int_{\Rd} \frac{u(z) }{ 1 + \abs z^{d} (\varphi({1}/{\abs z}))^{-1} } \D{z} \right) - C_2 t\,,
\end{align*}
for some constants $C_1, C_2$.
Combining we get
$$
(C+C_2)\inf_{\sB_{\frac{1}{2}}} u \geq (C+C_2) t \geq - C_0 + C_1 \left( \int_{\Rd} \frac{u(z) }{  1 + \abs z^{d + \alpha} } \D{z}  +  \int_{\Rd} \frac{u(z) }{ 1 + \abs z^{d} (\varphi({1}/{\abs z}))^{-1} } \D{z}\right)\,,
$$
and the result follows.
\end{proof}

Using \cref{T5.2,L6.1} we obtain the following
\begin{lemma}\label{L6.2}
Let $\Omega \subset \Rd$ be any open set. Suppose there exists
$x_0 \in \sB_{\frac{1}{2}}$ and $ \varrho > 0$ such that 
$\sB_{2 \varrho }(x_0) \subset \Omega \cap \sB_{\frac{1}{2}}$. Denote by $D = \sB_{ \varrho }(x_0)$.
Let $u \in \cC(B_1)$ be a viscosity solution of
\begin{equation*}
\begin{split}
 \tilde\cM^{+} u \geq - C_0    \quad \text{and } \quad  & \tilde\cM^{-} u \leq C_0 \quad  \text{ in} 
 \quad \sB_1 \cap \Omega,
 \\
    & \quad  u = 0   \quad \quad  \text{ in} \quad \sB_1 \setminus \Omega\,.
  \end{split}
\end{equation*}
Assume in addition, that $u \geq 0$ in $\Rd$. Then
$$
\sup_{\sB_{\frac{3}{4}}} u \leq\, C \left( \inf_{D} u + C_0 \right),
$$
where constant $C$ depends only on $(d, \lambda, \Lambda, \alpha, \varphi, \varrho )$.
\end{lemma}

\begin{proof}
Since $u \geq 0$ in $\sB_1$ and $\tilde\cM^{+}u \geq -C_0$ in $\sB_1 \cap \left\lbrace u > 0\right\rbrace$, we
have $\tilde\cM^{+}u \geq -C_0$ in all of $\sB_1$. Thus, by \cref{T5.2} and a standard
covering argument, we have
$$
 \sup_{\sB_{\frac{3}{4}}} u \leq\,  C \left( \int_{\Rd} \frac{ u(y)}{1 + \abs y^{d + \alpha}}\D{y}
  + \int_{\Rd} \frac{ u(y)}{1 + \abs y^d (\varphi ({1}/{\abs y}))^{-1}} \D{y} + C_0 \right)\,.
$$
Again, using \cref{L6.1} in the ball $\sB_{2 \varrho }(x_0)$, we get
$$
\int_{\Rd} \frac{ u(y)}{1 + \abs y^{d + \alpha}} \D{y} + \int_{\Rd} \frac{ u(y)}
{1 + \abs y^d (\varphi({1}/{\abs y}))^{-1}} \D{y} \leq C \left( \inf_{D} u + C_0 \right)\,,
$$
where $D = B_{ \varrho }(x_0)$. Combining the previous estimates, the result follows.
\end{proof}

Finally, we prove \cref{T6.1}
\begin{proof}[Proof of \cref{T6.1}]
Proof follows by following the arguments of \cite[Theorem~1.2]{RS19} and using \cref{L6.1,L6.2}.
\end{proof}

%%%%%%%%%%%%%%%%%%%%%%%%%%%%%%%%%%%%%%%%%%%%%%%%%%%%%%%%%%%%%%%%%%%%%%%%%%%%%%%%
\section*{Acknowledgements}
The research of Anup Biswas was supported in part by DST-SERB grants EMR/2016/004810 and MTR/2018/000028.
Mitesh Modasiya is partially supported by CSIR PhD fellowship (File no. 09/936(0200)/2018-EMR-I ).

%%%%%%%%%%%%%%%%%%%%%%%%%%%%%%%%%%%%%%%%%%%%%%%%%%%%%%%%%%%%%%%%%%%%%%%%%%%%%%%%
%%%%%%%%%%%%%%%%%%%%%%%%%%%%%%%%%%%%%%%%%%%%%%%%%%%%%%%%%%%%%%%%%%%%%%%%%%%%%%%%
%\bibliographystyle{plain}      % mathematics and physical sciences
%
%\bibliography{}
%%\bibliography{}

\end{document}